\documentclass[twoside,a4paper,11pt,intlimits]{elsarticle}

\usepackage{amssymb}
\usepackage{amsmath}
\usepackage{amssymb}
\usepackage{amsthm}

\usepackage{enumitem}
\setenumerate{label={\rm (\alph{*})}}

\usepackage[frak,operator]{paper_diening}

\numberwithin{equation}{section} 

\newtheorem{theorem}{Theorem}[section]
\newtheorem{lemma}[theorem]{Lemma}

\newtheorem{corollary}[theorem]{Corollary}

\newtheorem{definition}[theorem]{Definition}

\newtheorem{remark}[theorem]{Remark}

\numberwithin{equation}{section}

\newcommand{\bfve}{\bfv_\epsilon}
\newcommand{\bfvm}{\bfv^{m}}
\newcommand{\bfom}{\bfomega}

\newcommand{\bfoe}{\bfomega_\epsilon}
\newcommand{\bfon}{\bfomega^m}
\newcommand{\nl}{\bfomega\times\bfv}
\newcommand{\nlm}{\bfomega^m\times\bfv^m}
\newcommand{\nlmn}{\bfomega^m_{n}\times\bfv^m_{n}}
\newcommand{\nle}{\bfoe\times\bfve}

\providecommand{\Bog}{\mathrm{Bog}}
\providecommand{\curl}{\mathrm{curl}\,}
\providecommand{\divergence}{\mathrm{div}}
\providecommand{\Div}{\divergence\,}

\providecommand{skp}[2]{{\langle{#1},{#2}\rangle}}
 \providecommand{\fdg}{{\,\big|\,}}

\providecommand{\R}{\setR}
\providecommand{\N}{{\mathbb{N}}}
\providecommand{\Z}{{\mathbb{Z}}}
\providecommand{\ep}{\curl}
\newcounter{formel}

\providecommand{\Bog}{\ensuremath{\text{\rm Bog}}}

\providecommand{\Rdr}{{\setR^3}}

\begin{document}

%%%%%%%%%%%%%%%%%%%%%%%%%%%%%%%%%%%%%%%%%%%%%%%%%%%%%%%%%%%%%%%%%%%%%%%%%%%%%%%%
%%%
%%% begin document
%%%
\begin{frontmatter}

  \title{\textbf{On the existence of weak solutions for the steady Baldwin-Lomax model and
    generalizations}\footnote{Dedicated to the memory of Christian G. Simader}} 

  \author{Luigi C.~Berselli\corref{cor1}} \ead{luigi.carlo.berselli@unipi.it}

  \address{Dipartimento di Matematica, Universit\`a di Pisa, Via F.~Buonarroti 1/c,
    I-56127 Pisa, Italy}

\author{Dominic~Breit} \ead{d.breit@hw.ac.uk}

\address{Heriot-Watt University, Department of Mathematics, Riccarton, EH14 4AS Edinburgh}

\begin{abstract}
  In this paper we consider the steady Baldwin-Lomax model, which is a rotational model
  proposed to describe turbulent flows at statistical equilibrium. The Baldwin-Lomax model
  is specifically designed to address the problem of a turbulent motion taking place in a
  bounded domain, with Dirichlet boundary conditions at solid boundaries. The main
  features of this model are the degeneracy of the operator at the boundary and a
  formulation in velocity/vorticity variables. The principal part of the operator is
  non-linear and it is degenerate, due to the presence (as a coefficient) of a power of
  the distance from the boundary: This fact makes the existence theory naturally set in
  the framework of appropriate weighted-Sobolev spaces. 
  % This model is particular relevant to study the generation of vorticity in turbulent
  % fluids.
\end{abstract}
\end{frontmatter}
\section{Introduction}
\label{sec:intro}
In this paper we study the model (and some of its 
variants which are interesting from the mathematical point of view)  introduced by Balwdin and
Lomax~\cite{BL1978}%\marginpar{Why don't we write $\Delta\bfv$ instead of $\Div\bfD\bfv$?}
\begin{equation}
  \label{eq:Balwdin-Lomax}
  \left\{\begin{aligned} 
      -\nu_{0}\,\Div\bfD\bfv+(\nabla\bfv)\bfv +\curl\big(d^2|\curl\bfv|\curl\bfv\big)
      +\nabla \pi&=\bff& \mbox{in        $\Omega$,}
      \\
      \Div \bfv&=0& \mbox{in $\Omega$,}
      \\
      \bfv&=\mathbf{0}& \mbox{on $\partial\Omega$,}
    \end{aligned}\right.
\end{equation}
to describe turbulent fluids at the statistical equilibrium, where $d$ is the distance
from the boundary. We recall that, starting from the work of O.~Reynolds in the 19th
Century, a classical paradigm is that of decomposing the velocity into the sum of a mean
part and (turbulent) fluctuations, see~\cite{BIL2006}. One basic question is how to model
the effect of the smaller scales on the larger ones. The Boussinesq assumption suggests
that --in average--this produces an additional turbulent viscosity $\nu_{T}$, which is
proportional to the mixing length and to the kinetic energy of fluctuations (at least in
the Kolmogorov-Prandtl approximation). In the analysis of Baldwin and Lomax, this leads to
a turbulent viscosity of the form
\begin{equation*}
  \nu_{T}(\bfv(\bfx))\sim \ell^{2}(\bfx)|\curl\bfv(\bfx)|,
\end{equation*}
where $\ell$ is a multiple of the distance from the boundary and $\curl
\bfv=\nabla\times\bfv$, hence arriving to the model~\eqref{eq:Balwdin-Lomax} when the
equations for the turbulent flow are considered in the rotational formulation. (Mean velocities
denoted from now on as $\bfv$).

The Baldwin-Lomax model~\eqref{eq:Balwdin-Lomax} has been recently revisited --in the
unsteady case-- by Rong, Layton, and Zhao~\cite{RLZ2019}, in order to take into account
also of the effects of back-scatter. This involves, in addition to the usual time
derivative $\frac{\partial\bfv}{\partial t}$, a dispersive term of the form
\begin{equation*}
  \curl\Big(\ell^{2}(\bfx)\,\curl\frac{\partial\bfv}{\partial t}(t,\bfx)\Big),
\end{equation*}%\marginpar{I don't get this sentence}
resembling that appearing in Kelvin-Voigt materials. Also in this case the problem has
some degeneracy at the boundary. Different mathematical tools are required to handle the
above term: being of the Kelvin-Voigt type, the latter differential operator is linear and
not dissipative, but instead it is dispersive. Further details, and its analysis in
connection with Turbulent-Kinetic-Energy (TKE) models are studied in~\cite{ABLN2019}, in
the case of a turbulent viscosity depending only on the turbulent kinetic energy, but not
on $\curl\bfv$. Related results involving a selective anisotropic turbulence model can be
found also in~\cite{CAB2020}. 

Here, we consider --as a starting point-- the problem at statistical equilibrium. We study
just the steady case, which contains nevertheless several peculiar properties; the methods
and techniques involved are rather different than those used in the previous mathematical
theory of unsteady Baldwin-Lomax type models in~\cite{ABLN2019,RLZ2019}.

The class of problems we study is that of finding a velocity field
$\bfv:\Omega\rightarrow\R^3$ and a pressure function $\pi:\Omega\rightarrow\R$ such that
the following boundary value problem for a nonlinear system of partial differential
 equations is satisfied
 \begin{align*}
   \left\{
     \begin{aligned}%{c c}
       -\nu_{0}\,\Div\bfD\bfv+   \curl\bfS+(\nabla\bfv)\,\bfv&= -\nabla \pi+\bff& \mbox{in
         $\Omega$,}
       \\
       \Div \bfv&=0\qquad& \mbox{in $\Omega$,}
       \\
       \bfv&=\mathbf{0}\qquad\quad& \mbox{ \,on $\partial\Omega$.}
     \end{aligned}
   \right.
 \end{align*}
 Here $\Omega$ denotes a bounded smooth domain in $\R^3$, and $\bff:\Omega\rightarrow\R^3$
 is the volume force and $\nu_{0}\geq0$ is the kinematic viscosity.
 
As a generalized Baldwin-Lomax model, we will also consider  the
stress tensor $\bfS:\, \Omega\rightarrow\R^{3\times 3}$  given by
\begin{align*}
  %\label{eq:S}
\bfS=\bfS(\bfx,\ep(\bfv))=d(\bfx)^\alpha(\kappa+|\curl(\bfv)|)^{p-2}\curl(\bfv),
\end{align*}
where $d(\bfx)=\mathrm{dist}(\bfx,\partial\Omega)$, and $\alpha>0$, $p>1$, $\kappa\geq0$ 
are given constants.

We will prove the existence of weak solutions for various models and highlight the role of
the parameters $p,\,\alpha$, and $\nu_{0}$. The analysis requires substantial changes in
the mathematical approach depending on the range of these constants.

The main result we prove is the existence of weak solution in appropriate (weighted) function
spaces. The results are obtained by using a classical Galerkin approximation procedure and
the passage to the limit is done by means of monotonicity and truncation methods typical
of the analysis of non-Newtonian fluids, see for instance the reviews
in~\cite{Bre2017,Ruz2013}.
\section*{Acknowledgments}
The research of Luigi C. Berselli that led to the present paper was partially supported
by a grant of the group GNAMPA of INdAM and by the project of the University of Pisa
within the grant PRA$\_{}2018\_{}52$~UNIPI \textit{Energy and regularity: New techniques
  for classical PDE problems.} 

The authors are grateful to R.~Lewandowski for having introduced us to the topic and suggested
to study the mathematical properties of the Baldwin-Lomax model.

\section{Functional setting}
In the sequel $\Omega\subset\R^{3}$ will be a smooth and bounded open set, as usual we
write $\bfx=(x_{1},x_{2},x_{3})=(x',x_{3})$ for all $\bfx\in \R^{3}$.  In particular, we
assume that the boundary $\partial\Omega$ is of class $C^{0,1}$, such that the normal unit
vector $\bfn$ at the boundary is well defined and other relevant properties hold true. We
recall a domain is of class $C^{k,1}$ if for each point $P\in\partial\Omega$ there are
local coordinates such that in these coordinates we have $P=0$ and $\partial\Omega$ is
locally described by a $C^{k,1}$-function, i.e.,~there exist $R_P,\,R'_P \in
(0,\infty),\,r_P\in (0,1)$ and a $C^{k,1}$-function $a_P:B_{R_P}^{2}(0)\to B_{R'_P}^1(0)$
such that
\begin{itemize}
\item[i)] $\bfx\in \partial\Omega\cap (B_{R_P}^{2}(0)\times B_{R'_P}^1(0))\
  \Longleftrightarrow \ x_3=a_P(x')$,
\item[ii)] $\Omega_{P}:=\{(x\in \R^{3}\fdg x'\in B_{R_P}^{2}(0),\
  a_P(x')<x_3<a_P(x')+R'_P\}\subset \Omega$,
\item [iii)] $\nabla a_P(0)=\bfzero,\text{ and }\forall\,x'  \in B_{R_P}^{2}(0)\quad
  |\nabla a_P(x')|<r_P$, 
\end{itemize}
where $B_{r}^k(0)$ denotes the $k$-dimensional open ball with center
$0$ and radius ${r>0}$.  

We also define the distance $d(\bfx,A)$ of a point from a closed set
$A\subset\R^{3}$ as follows 
\begin{equation*}
  d(\bfx,A):=\inf_{\bfy\in A}|\bfx-\bfy|,
\end{equation*}
and we denote by $d(\bfx)$ the distance of $\bfx$ from the boundary of $\Omega$
\begin{equation*}
  d(\bfx):=d(\bfx,\partial\Omega).
\end{equation*}
We recall a well-known lemma about the distance function $d(\bfx)$, see for instance
Kufner~\cite{Kuf1985}.

\begin{lemma}
\label{lem:distance}
  Let $\Omega$ be a domain of class $C^{0,1}$, then there exist constants
  $0<c_{0},c_{1}\in\R$ such that 
  \begin{equation*}
    c_{0}\,d(\bfx)\leq|a(x')-x_{3}|\leq c_{1}\,d(\bfx)\qquad      \forall\,\bfx=(x',x_{3})\in
    \Omega_{P}. 
  \end{equation*}
\end{lemma}
For our analysis we will use the customary Lebesgue $(L^{p}(\Omega),\|\,.\,\|_{p})$ and
Sobolev spaces $(W^{k,p}(\Omega),\|\,.\,\|_{k,p})$ of integer index $k\in \N$ and with
$1\leq p\leq\infty$. We do not distinguish scalar and vector valued spaces, we just use
boldface for vectors and tensors.  We recall that $L^{p}_{0}(\Omega)$ denotes the subspace
with zero mean value, while $W^{1,p}_{0}(\Omega)$ is the closure of the smooth and
compactly supported functions with respect to the $\|\,.\,\|_{1,p}$-norm. If $\Omega$ is
bounded and if $1<p<\infty$, the following two relevant inequalities hold true:

\noindent 1) the \Poincare{} inequality
\begin{equation}
  \label{eq:Poincare}
  \exists\, C_{P}(p,\Omega)>0:\qquad \|\bfu\|_{p}\leq C_{P}\|\nabla
  \bfu\|_{p}\qquad\forall\,\bfu\in W^{1,p}_{0}(\Omega);
\end{equation}
2) the  Korn inequality
\begin{equation}
  \label{eq:Korn}
  \exists\, C_{K}(p,\Omega)>0:\qquad \|\nabla\bfu\|_{p}\leq C_{K}\|\bfD
  \bfu\|_{p}\qquad\forall\,\bfu\in W^{1,p}_{0}(\Omega),
\end{equation} 
where $\bfD\bfu$ denotes the symmetric part of the matrix of derivatives $\nabla\bfu$. 

As a combination of~\eqref{eq:Poincare}-\eqref{eq:Korn} we also have that for $1\leq p<3$
the Sobolev-type inequality
\begin{equation}
  \label{eq:Sobolev}
  \exists\, C_{S}>0:\qquad \|\bfu\|_{p^{*}}\leq C_{S}\|\bfD \bfu\|_{p},
\end{equation} 
holds true for all $\bfu\in W^{1,p}_{0}(\Omega)$, where $p^{*}:=\frac{3p}{3-p}$.

The Korn inequality allows to control the full gradient in $L^p$ by its symmetric part,
for functions which are zero at the boundary. Classical results (cf.~Bourguignon and
Brezis~\cite{BB1974}) concern controlling the full gradient with curl \&{} divergence. The
following inequality holds true: For all $s\geq1$ and $1<p<\infty$, there exists a
constant $C=C(s,p,\Omega)$ such that,
\begin{equation*}
  \|\bfu\|_{s,p}\leq C\Big[\|\Div \bfu\|_{s-1,p}+\|\curl
  \bfu\|_{s-1,p}+\|\bfu\cdot \bfn\|_{s-1/p,p,\partial\Omega}+\|\bfu\|_{s-1,p}\Big],
\end{equation*}
for all $\bfu\in (W^{s,p}(\Omega))^3$, where $\|\,.\,\|_{s-1/p,p,\partial\Omega}$ is the trace
norm as explained below. This same result has been later improved by von Wahl~\cite{vWah1992} obtaining,
under geometric conditions on the domain, the following estimate without lower order
terms: Let $\Omega$ be such that $b_1(\Omega)=b_2(\Omega)=0$, where $b_i(\Omega)$ denotes
the i-th Betti number, that is the the dimension of the i-th homology group $H^i(\Omega,
\Z)$. Then, there exists $C=C(p,\Omega)$ such that
\begin{equation}
  \label{eq:div-curl}
  \|\nabla \bfu\|_{p}\leq C\big(\|\Div \bfu\|_p+\|\curl \bfu\|_p\big),
\end{equation}
for all $\bfu\in (W^{1,p}(\Omega))^3$ satisfying either $(\bfu\cdot\bfn)_{|\partial\Omega}=0$ or
$(\bfu\times \bfn)_{|\partial\Omega}=0$. 

In the trace-norm fractional derivative appear in a natural way. Nevertheless, we need also
to handle fractional spaces $W^{r,p}(\Omega)$, which are defined by means of the semi-norm
\begin{equation*}
  [u]^{p}_{s,p}:=\int_{\Omega}\int_{\Omega}\frac{|u(\bfx)-u(\bfy)|^{p}}{|\bfx-\bfy|^{3+s p}}\,\mathrm{d}\bfx
  \mathrm{d}\bfy\qquad\text{for } 0<s<1 ,
\end{equation*}
as made by functions $u\in W^{[r],p}(\Omega)$, such that
$[D^{\alpha}u]_{r-[r],p}=[D^{\alpha}u]_{s,p}<\infty$, for all multi-indices $\alpha$ such
that $|\alpha|=[r]$ (for the trance norm one has to integrate instead with respect to the
2-dimensional Hausdorff measure). The main result we need is the following generalization
of the classical Hardy inequality: Let $u\in L^{p}(\Omega)$, then
  \begin{equation}
    \label{eq:characterization-fractional}
    \frac{u}{d^{s}}\in L^{p}(\Omega)\Longleftrightarrow u\in W^{s,p}_{0}(\Omega)\quad
\text{for all }    0<s<1,\text{ with }     s-\frac{1}{p}\not=\frac{1}{2}.
  \end{equation}
\subsection{Weighted spaces}
\label{sec:weighted}
Since we have a boundary value problem with an operator which is space dependent, a
natural functional setting would be that of weighted Sobolev spaces. For this reason we
define now the relevant spaces we will use. We follow the notation from
Kufner~\cite{Kuf1985}.

We start by defining weighted Sobolev spaces. Let
$w(\bfx):\,\Omega\to\R^{+}$ be given a function (weight) which is non-negative and a.e. everywhere positive. We define,
for $1\leq p<\infty$, the weighted space $L^{p}(\Omega,w)=L^{p}_{w}(\Omega)$ as follows
\begin{equation*}
  L^{p}(\Omega,w):=\left\{\bff:\ \Omega\to\R^{n} \text{ measurable: }
    \int_{\Omega}|\bff(\bfx)|^{p}\, w(\bfx)\,\mathrm{d}\bfx<\infty\right\} .
\end{equation*} 
The definition is particularly relevant if it allows to work in the standard setting of
distributions $\mathcal{D}'(\Omega)$: for $p>1$ we have
\begin{equation*}
%  \label{eq:distributions}
  w^{-1/(p-1)}\in L^{1}_{loc}(\Omega)\quad\Rightarrow\quad  L^{p}(\Omega,w)\subset
  L^{1}_{loc}(\Omega)\subset \mathcal{D}'(\Omega).
\end{equation*}
It turns out that
$C^{\infty}_{0}(\Omega)$ is dense in $L^{p}(\Omega,w)$ if the weight satisfies at least
$w\in L^{1}_{loc}(\Omega)$, see~\cite{Kuf1985}.  In addition, $L^{p}(\Omega,w)$ is a
Banach space when equipped with the norm
\begin{equation*}
  \|\bff\|_{p,w}:=\Bigg(\int_{\Omega}|\bff(\bfx)|^{p}w(\bfx)\,\mathrm{d}\bfx\Bigg)^{1/p}. 
\end{equation*}
Clearly if $w(\bfx)\equiv 1$ then $L^{p}(\Omega,w)=L^{p}(\Omega)$. 

\smallskip 

Next, we define weighted Sobolev spaces
\begin{equation*}
  W^{k,p}(\Omega,w):=\left\{\bff:\ \Omega\to\R^{n}: D^{\alpha}\bff\in
    L^{p}(\Omega,w)\text{ for all }\alpha \text{ s.t. } |\alpha|\leq k\right\} ,
\end{equation*} 
equipped with the norm 
\begin{equation*}
  \|\bff\|_{k,p,w}:=\Bigg(\sum_{|\alpha|\leq k}\|D^{\alpha}\bff\|_{p,w}^{p}\Bigg)^{1/p}.
\end{equation*}
As expected, we define $W^{k,p}_{0}(\Omega,w)$ as follows
\begin{equation*}
  W^{k,p}_{0}(\Omega,w):=\overline{\left\{\bfphi\in
      C^{\infty}_{0}(\Omega)\right\}}^{\|\,.\,\|_{k,p,w}}. 
\end{equation*}

\medskip

In our application the weight $w(\bfx)$ will a power of the distance $d(\bfx)$ from the
boundary. Consequently, we specialize to this setting and give specific notions regarding
these so-called \textit{power-type weights}, see Kufner~\cite{Kuf1985}. First, it turns
out that $W^{k,p}(\Omega,d^{\alpha})$ is a separable Banach spaces provided $\alpha\in\R$,
$k\in\N$ and $1\leq p<\infty$. In this special setting, several results are stronger or
more precise due to the inclusion $L^{p}(\Omega,d^{\alpha})\subset L^{p}_{loc}(\Omega)$
for all $\alpha\in \R$.

Probably one of the most relevant properties is the embedding
\begin{equation}
\label{eq:embedding-L1}L^{p}(\Omega,d^{\alpha})\subset L^{1}(\Omega)\quad\text{if}\quad
\alpha<p-1.
\end{equation}
It follows directly from H\"older's inequality as follows
\begin{equation*}
  \int_{\Omega}|f%(\bfx)
  |\,\mathrm{d}\bfx=\int_{\Omega}d^{\alpha/p}|f%(\bfx)
  |d^{-\alpha/p}d\bfx
  \leq\Big(\int_{\Omega}d^{\alpha}|f|^{p}d\bfx\Big)^{1/p} \Big(\int_{\Omega}d^{-\alpha
    p'/p}d\bfx\Big)^{1/p'},   
\end{equation*}
using that the latter integral is finite if and only if 
\begin{equation*}
  \frac{ \alpha\,p'}{p}=\frac{\alpha}{p-1}<1
\end{equation*}
by Lemma~\ref{lem:distance}. Moreover, as in \cite[Prop.~9.10]{Kuf1985} it follows that
the quantity $\int_{\Omega}d^{\alpha}|\nabla \bff|^{p}\,\mathrm{d}\bfx$ is an equivalent norm in
$W^{1,p}_{0}(\Omega,d^{\alpha})$,  provided that $\alpha<p-1$.
\begin{remark}
\label{rem:alpha1}
The above results explain the critical role of the power $\alpha=p-1$ and highlight the
fact that the original Balwdin-Lomax model is exactly that corresponding to the critical
exponent. For the applications we have in mind the value of $\alpha$ is not so strictly
relevant and in fact, following the same procedure as in~\cite{ABLN2019}, it also makes
sense to consider turbulent viscosity as follows
  \begin{equation}
    \label{eq:Kelvin-Voigt}
    \nu_{T}(\bfv( \bfx))=\ell_{0}\,\ell( \bfx)|\curl\bfv( \bfx)|,
  \end{equation}
  for some $\ell_{0}\in\R^{+}$.
\end{remark}
Appropriate versions of the Sobolev inequality~\eqref{eq:Sobolev} are valid also for
weighted Sobolev spaces:
\begin{lemma}
  \label{lem:Sobolev}
  There exists a constant $C=C(\Omega,\delta,p)$, such that 
  \begin{equation}
    \label{eq:gen-Sobolev}
    \Big\|u(\bfx)-\dashint_\Omega u(\bfy)\,\mathrm{d}\bfy\Big\|_q\leq
    C\|d^\delta(\bfx) \nabla u(\bfx)\|_p=\|\nabla u \|_{p,d^{p \delta}},
  \end{equation}
  for all $u\in
    W^{1,p}(\Omega,d^{\delta p})$, where $q\leq \frac{3p}{3-p(1-\delta)}$.
\end{lemma}
For a proof see Hurri-Syrj\"{a}nen~\cite{Hur1994}. Note that this inequality is formulated
removing constants by means of subtracting averages and that the exponent $q$ equals to
$p^*$ if $\delta=0$.  This will be used later on to make a proper sense of the quadratic
term in the Navier--Stokes equations, cf.~ Definitions~\ref{eq:BL-steady}
and~\ref{def:weak3}.

In addition to~\eqref{eq:embedding-L1} and the Hardy inequality, the critical role of the
exponent $p-1$ is also reflected in results about general weights and their relation with
the maximal function.
\begin{definition}
  We say that $w\in L^{1}_{loc}(\R^{3})$, which is $w\geq0$ a.e., belongs to the Muckenhoupt
  class $A_{p}$, for $1<p<\infty$, if there exists $C$ such that
  \begin{equation*}
    \sup_{Q\subset \R^{n}}\Bigg(\dashint_{Q}w(\bfx)\,\mathrm{d}\bfx\Bigg)\Bigg(\dashint_{Q}
    w(\bfx)^{1/(1-p)}\,\mathrm{d}\bfx\Bigg)^{p-1}\leq C,
  \end{equation*}
where $Q$ denotes a cube in $\R^{3}$.
\end{definition}
The role of the power $\alpha$ will be crucial in the sequel and we recall the following
result, which allows us to embed the results within a classical framework and also to use
fundamental tools of harmonic analysis. The powers of the distance function belong to the
class $A_{p}$ according to the following well-known result (For a proof see for instance
Dur\'an, Sammartino, and Toschi~\cite[Thm.~3.1]{DST2008})
\begin{lemma}
  The function $w(\bfx)=\big(d(\bfx)\big)^{\alpha}$ is a Muckenhoupt weight of class
  $A_{p}$ if and only if $-1<\alpha<p-1$.
\end{lemma}
The main result which we will use about singular integrals, which follows from the
pioneering work of Muckenhoupt on maximal functions, is the following.
\begin{lemma}
\label{lem:Muckenhoupt}
  Let $CZ:\,C^{\infty}_{0}(\R^{n})\to C^{\infty}_{0}(\R^{n})$ be a standard
  Calder\'on-Zygmund singular integral operator. Let $w\in A_{p}$, for $1<p<\infty$. Then,
  the operator 
  $CZ$ is continuous from $L^{p}(\Omega,w)$ into itself.
\end{lemma}
We will use this result mainly on the operators related to the solution of the Poisson
equation, to reconstruct a vector field from its divergence and its curl.
\subsection{Solenoidal spaces}
As usual in fluid mechanics, when working with incompressible fluids, it is natural to
incorporate the divergence-free constraint directly in the function spaces. These spaces
are built upon completing the space of solenoidal smooth functions with compact support,
denoted as $\bfphi\in C^{\infty}_{0,\sigma}(\Omega)$, in an appropriate topology. For
$\alpha>0$ define
\begin{equation*}
  \begin{aligned}
   & L^{p}_{\sigma}(\Omega,d^{\alpha}):=\overline{\left\{\bfphi\in
        C^{\infty}_{0,\sigma}(\Omega)\right\}}^{\|\,.\,\|_{p,d^{\alpha}}},
    \\
    &W^{1,p}_{0,\sigma}(\Omega,d^{\alpha}):=\overline{\left\{\bfphi\in
        C^{\infty}_{0,\sigma}(\Omega)\right\}}^{\|\,.\,\|_{1,p,d^{\alpha}}}\,.
  \end{aligned}
\end{equation*}
For $\alpha=0$ they reduce to the classical spaces $L^{p}_{\sigma}(\Omega)$ and
$W^{1,p}_{0,\sigma}(\Omega)$.  Next, we will extensively use the following extension of
inequality~\eqref{eq:div-curl}.
\begin{lemma}
  \label{lem:positivity}
  Let $1<p<\infty$ and assume that the weight $w$ belongs to the class $A_p$. Then
  there exists\footnote{The space $W^{1,p}_0(\Omega,w)$ can be
  replaced by other function spaces, where $C^{\infty}_{0}(\Omega)$ functions are dense.} a constant $C(\Omega,w)$ such that
\begin{equation*}
  \|\nabla \bfu\|_{p,w}\leq C(w,\Omega)(  \|\Div \bfu\|_{p,w}+
  \|\curl \bfu\|_{p,w})\qquad \forall\,\bfu\in W^{1,p}_0(\Omega,w).
\end{equation*}

\end{lemma}
\begin{proof}
  For vector fields $\bfu$ with compact support in $\Omega$ we have the well-known
  identity
\begin{equation*}
  \curl\curl\bfu=-\Delta\bfu+\nabla\Div\bfu.
\end{equation*}
 By use of the Newtonian potential this implies
  \begin{equation*}
   \bfu(\bfx)=-\frac{1}{4\pi}\nabla_{\bfx}\int_{\Omega}\frac{\divergence_{\bfy}\,
     \bfu(\bfy)}{|\bfx-\bfy|}\,\mathrm{d}\bfy+\frac{1}{4\pi}\curl_{\bfx}\int_\Omega\frac{\curl_{\bfy}
      \bfu(\bfy)}{|\bfx-\bfy|}\,\mathrm{d}\bfy.
  \end{equation*}
  Hence, we obtain
  \begin{equation*}
    \begin{aligned}
      \nabla
      \bfu(\bfx)&=-\frac{1}{4\pi}\nabla_{\bfx}\,\nabla_{\bfx}\int_{\Omega}
      \frac{\divergence_{\bfy}\,\bfu(\bfy)}{|\bfx-\bfy|}\,\mathrm{d}\bfy 
      +\frac{1}{4\pi}\nabla_{\bfx}\,\curl_{\bfx}\int_\Omega\frac{\curl_{\bfy} \,
        \bfu(\bfy)}{|\bfx-\bfy|}\,\mathrm{d}\bfy,
      \\
      &=-\frac{1}{4\pi}\int_{\Omega} \nabla_{\bfx}\,\nabla_{\bfx}
      \frac{\Div_{\bfy}\,\bfu(\bfy)} {|\bfx-\bfy|}
      \,\mathrm{d}\bfy+\frac{1}{4\pi}\int_\Omega\nabla_{\bfx}\,\curl_{\bfx}\frac{\curl_{\bfy}\,
        \bfu(\bfy)}{|\bfx-\bfy|}\,\mathrm{d}\bfy,
      \\
      &=CZ_{1}[\divergence \,\bfu](\bfx)+CZ_{2}[\curl \bfu](\bfx),
    \end{aligned}
  \end{equation*}
  where both terms $CZ_i$ from the right-hand side are Calderon-Zygmund type singular
  integrals. Applying the Muckenhoupt result from Lemma~\ref{lem:Muckenhoupt} the claim follows.
\end{proof}
In particular, we will use the latter result in the following special form
\begin{corollary}
  For $-1<\alpha<p-1$ there exists a constant $C=C(\Omega,\alpha,p)$ such that
\begin{equation}
  \label{eq:grad-curl-weighted}
  \int_\Omega d^\alpha|\nabla\bfv|^p\,\mathrm{d}\bfx\leq C  \int_\Omega
  d^\alpha|\curl\bfv|^p\,\mathrm{d}\bfx\qquad\forall\, \bfv\in W^{1,p}_{0,\sigma}(\Omega,d^\alpha). 
\end{equation}
\end{corollary}
A basic tool in mathematical fluid mechanics is the construction of a continuous right
inverse of the divergence operator with zero Dirichlet conditions. An explicit
construction is due to the \Bogovskii{} and it is reviewed in
Galdi~\cite[Ch.~3]{Gal2011}. The following results holds true
\begin{theorem}
\label{thm:bog}
  Let $\omega\subset\R^3$ be a bounded smooth domain and let $f\in L^{p}_0(\omega)$
  there exists at least one $\bfu=\Bog_\omega(f)\in W^{1,p}_0(\omega)$ which solves the
  boundary value problem
  \begin{equation*}
    \left\{  
      \begin{aligned}
        \Div \bfu&=f\qquad&\text{in }\omega,
        \\
        \bfu&=\mathbf{0}\qquad&\text{on }\partial\omega.
      \end{aligned}
    \right.
  \end{equation*}
  Among other spaces, the operator $Bog_\omega$ is linear and continuous from
  $L^p(\omega)$ to $W^{1,p}_0(\omega)$, for all $p\in (1,\infty)$
\end{theorem}
\subsection{Solenoidal Lipschitz truncation}
\label{sec:sol}

We recall that the nonlinear operator defined as follows
$\mathcal{A}_{p}$ 
\begin{equation*}
\mathcal{A}_{p}\bfw=  -\Div\big(|\bfD\bfw|^{p-2}\bfD\bfw\big),
\end{equation*}
is strongly monotone in $W^{1,p}_{0,\sigma}(\Omega)$, for $1<p<\infty$. In fact
\begin{equation*}
%  \int_\Omega
  (|\bfD\bfw_1|^{p-2}\bfD\bfw_1-|\bfD\bfw_2|^{p-2}\bfD\bfw_2)
  :(\bfD\bfw_1-\bfD\bfw_2)%\,\mathrm{d}\bfx
  \geq0,
\end{equation*}
with equality if and only if $\bfD\bfw_1=\bfD\bfw_2$. A crucial point in the classical
Minty-Browder argument relies on analyzing, for $\bfv_n,\bfv\in W^{1,p}(\Omega)$ the
non-negative quantity
\begin{equation*}
    \int_\Omega(|\bfD\bfv_{n}|^{p-2}\bfD\bfv_{n}-|\bfD\bfv|^{p-2}\bfD\bfv):(\bfD\bfv_{n}-
    \bfD\bfv)\,\mathrm{d}\bfx\geq0.
\end{equation*}
Here $\bfv_{n}$ is a Galerkin approximation and $\bfv$ its weak $W^{1,p}_{0,\sigma}$-limit.
Using the weak formulation for both $\bfv_{n}$ and its limit $\bfv$ one can show (using the
monotonicity argument) that
\begin{equation*}
  \mathcal{A}_p(\bfv_{n})\to\mathcal{A}_p(\bfv)\quad\text{at least in}\quad
  (C^\infty_{0,\sigma}(\Omega))'. 
\end{equation*}
Two main points of the classical argument are 1) being allowed to use $\bfv_{n}$ as test function and 2)
showing that
\begin{equation*}
  \int_\Omega(\nabla\bfv_{n})\,\bfv_{n}\cdot\bfv_{n}\,\mathrm{d}\bfx\to
  \int_\Omega(\nabla\bfv)\,\bfv\cdot\bfv\,\mathrm{d}\bfx.
\end{equation*}
In general item 1) trivially follows for all $1<p<\infty$, due to the continuity of the
operator $\mathcal{A}_{p}$. We will see that this point is not satisfied with the
degenerate operators we handle in Section~\ref{sec:Classical-BL}-\ref{sec:Generalised-BL}
and appropriate localization/regularization/truncation must to be introduced, see
below. Hence, we are using here some known technical tools in a new and non-standard
context: the use of local techniques is not motivated by the presence of the convective
term, but by the character of the nonlinear stress-tensor. Probably our analysis can be
extended also to other degenerate fractional operator as those studied by Abdellaoui,
Attar, and Bentifour~\cite{AAB2019}.

Note also that it is for the request 2) that a limitation on the exponent arises, since
$\bfv_{n}\to\bfv$ in $L^q$ for $q<p^*=\frac{3p}{3-p}$ and this enforces a lower bound on
the allowed values of $p$. In the analysis of non-Newtonian fluid this classical
monotonicity argument is not applicable when $p\leq\frac{9}{5}$ (in the steady case). To
overcome this problem and to solve the system also for smaller values of $p$ (up to
$\frac{6}{5}$) one needs test functions which are Lipschitz continuous, hence one needs to
properly truncate $\bfv^m-\bfv$. This is the point where the Lipschitz truncation,
originally developed by Acerbi and Fusco~\cite{AcFu1,AcFu2} in the context of quasi-convex
variational problems, comes into play. In fluid mechanics this tool has been firstly used
in~\cite{DMS,FMS}, for a review we refer to~\cite{Bre2017,Ruz2013}. Being strongly
nonlinear and also non-local, the Lipschitz truncation destroys the solenoidal character of
a given function. Consequently, the pressure functions has to be introduced.  Another
approach is that of constructing a divergence-free version of the Lipschitz truncation -
extending a solenoidal Sobolev function by a solenoidal Lipschitz function. This approach
has been developed in~\cite{BDF2012,BDS2013} and it completely avoids the appearance of
the pressure function and highly simplifies the proofs avoiding results obtained in
Simader and Sohr~\cite{SS1992} (as done in Diening, R{\r u}{\v{z}}i{\v{c}}ka, and
Wolf~\cite{DRW2010}) to associate a pressure to the weak solution.  We report the
following version which can be found in~\cite[Thm. 4.2]{BDS2013}.
\begin{theorem}
  \label{thm:remlip}
  Let $1< s< \infty$ and $B\subset\R^3$ a ball. Let $(\bfu^m) \subset W^{1,s}_{0,\sigma}(B)$
  be a weak $W^{1,s}_{0,\sigma}(B)$ null sequence extended by zero to $\R^3$. Then there exist
  $j_0 \in \mathbb{N}$ and a double sequence $(\lambda^{m,j}) \subset\R$ with
  $2^{2^j} \leq \lambda^{m,j}\leq 2^{2^{j+1}-1}$ a sequence of functions $(\bfu^{m,j})$
  and open sets%
\footnote{The set $\mathcal{O}^{m,j}$ is explicitly given by $\mathcal{O}^{m,j} := \set{\mathcal M(\nabla^2 (\curl^{-1} \bfu^m))>
      \lambda^{m,j}}$, where $\mathcal M$ is the Hardy-Littlewood maximal operator and
    $\curl^{-1}=\curl\Delta^{-1}$.}
$(\mathcal{O}^{m,j})$ with the
  following properties for $j \geq j_0$.
  \begin{enumerate}[leftmargin=1cm,itemsep=1ex,label={\rm (\alph{*})}]
  \item \label{itm:remlip1} 
$\bfu^{m,j}\in
    W^{1,\infty}_{0,\sigma}(2B)$ and $\bfu^{m,j}=\bfu^m$ on $\Rdr
    \setminus \mathcal{O}^{m,j}$ for all $m \in \setN$;
  \item \label{itm:remlip2} 
$\|\nabla\bfu^{m,j}\|_\infty\leq
    c\lambda^{m,j}$ for all $m \in \setN$;
  \item  \label{itm:remlip3} 
$\bfu^{m,j} \to 0$ for $m \to
    \infty$ in $L^\infty(\Omega)$;
  \item \label{itm:remlip4}
 $\nabla\bfu^{m,j} \weakastto 0$ for $m
    \to \infty$ in $L^\infty(\Omega)$;
  \item \label{itm:remlip5} 
For all $m,j \in \N$ it holds
    $\norm{\lambda^{m,j} \chi_{\mathcal{O}^{m,j}}}_s
    \leq c(s)\, 2^{-\frac{j}{s}}\norm{\nabla \bfu^m}_s$.
  \end{enumerate}
\end{theorem}
As usual we denote by $\chi_A$ the indicator function of the measurable set
$A\subset\R^3$.
\section{Existence of weak solutions for the  Baldwin-Lomax model in the steady case}
\label{sec:Classical-BL}
In this section we consider the model for the average of turbulent fluctuations attributed
to Baldwin and Lomax~\eqref{eq:Balwdin-Lomax}. By using a standard notation we denote the
curl of $\bfv$ by $\bfomega$
\begin{equation*}
  \bfomega=\curl\bfv=\nabla\times\bfv.
\end{equation*}
Since we consider the equations in a rotational setting, we write the convective term
as follows 
\begin{equation*}
      (\nabla\bfv)\bfv=\bfomega\times\bfv+\frac{1}{2}\nabla|\bfv|^{2}.
\end{equation*}
By redefining the pressure we can consider the following steady system for a turbulent
flow at statistical equilibrium
\begin{equation}
  \left\{\begin{aligned} 
      -\nu_{0}\,\Div\bfD\bfv+\nl+\curl\big(d^2|\bfomega|\bfomega\big)
      +\nabla \pi&=\bff& \mbox{in        $\Omega$,}
      \\
      \Div \bfv&=0\quad& \mbox{in $\Omega$,}
      \\
      \bfv&=\mathbf{0}\quad\quad& \mbox{ \,on $\partial\Omega$,}
\end{aligned}\right.\label{eq:BLM}
\end{equation}
in the case $\nu_{0}>0$.  We have the following result, which does not follow by the
standard theory of monotone operator.
\begin{theorem}
  \label{thm:easy}
  Let be given $\nu_{0}>0$ and $\bff\in W^{-1,2}(\Omega)=(W^{1,2}_0(\Omega))'$. Then,
there exists $\bfv\in W^{1,2}_{0,\sigma}(\Omega)$, with $\bfomega \in L^3(\Omega,d^2)\cap
L^3_{loc}(\Omega)$, which is a weak solution to~\eqref{eq:BLM}, that is such that
\begin{equation*}
  \int_{\Omega}\nu_{0}\,\bfD\bfv:\bfD\bfphi+d^{2}|\bfomega|\bfomega\cdot\curl\bfphi
  +(\nl)\cdot\bfphi\,\mathrm{d}\bfx  
  =\langle\bff,\bfphi\rangle\qquad\forall  
  \,\bfphi\in C^\infty_{0,\sigma}(\Omega).   
\end{equation*}
\end{theorem}
Here $\langle\cdot,\cdot\rangle$ denotes generically a duality pairing.  By density it is
enough to consider test functions $\bfphi\in W^{1,2}_{0,\sigma}(\Omega)$, with
$\curl\bfphi \in L^3_{loc}(\Omega)$.  The proof of Theorem~\ref{thm:easy} is based on a
Galerkin approximation and monotonicity arguments (beyond the classical Minty-Browder
trick) to pass to the limit

We observe that the term coming from Baldwin-Lomax approach is monotone too. We prove for
a general $p\in(1,\infty)$ and a general non-negative weight the following inequality.
\begin{lemma}
  For smooth enough $\bfomega_{i}$ (it is actually enough that $d^\frac{\alpha}{p}\bfomega_i\in
    L^{p}(\Omega)$, with $1<p<\infty$) and for $\alpha\in\R^+$  it holds that
    \begin{equation*}
      \int_\Omega
      (d^{\alpha}|\bfomega_1|^{p-2}\bfomega_1-d^{\alpha}|\bfomega_2|^{p-2}\bfomega_2)
\cdot(\bfomega_1-\bfomega_2)\,\mathrm{d}\bfx\geq0,
\end{equation*}
for any (not necessarily the distance) bounded function such that  $d:\Omega\to\R^+$ for
a.e. $\bfx\in\Omega$. 
\end{lemma}    
\begin{proof}
  We have
  \begin{equation*}
    \label{eq:lowerbound}
    \begin{aligned}
      &\int_\Omega
      (d^{\alpha}|\bfomega_1|^{p-2}\bfomega_1-d^{\alpha}|\bfomega_2|^{p-2}\bfomega_2)\cdot(\bfomega_1-\bfomega_2)\,\mathrm{d}\bfx
      \\
      &= \int_\Omega
      (|d^{{\frac{\alpha}{p}}}\bfomega_1|^{p-2}d^{{\frac{\alpha}{p}}}\bfomega_1
      -|d^{{\frac{\alpha}{p}}}\bfomega_2|^{p-2}d^{{\frac{\alpha}{p}}}\bfomega_2):(d^{{\frac{\alpha}{p}}}\bfomega_1
      -d^{{\frac{\alpha}{p}}}\bfomega_2)\,\mathrm{d}\bfx,
      % \\
      % &
      % \geq\int_\Omega|d^{{\frac{\alpha}{p}}}\bfomega_1-d^{{\frac{\alpha}{p}}}\bfomega_2|^p=\int_\Omega
      % d^\alpha\,|\bfomega_1-\bfomega_2|^p\geq0.
    \end{aligned}
  \end{equation*}
  where the last inequality derives from the same monotonicity/convexity argument used
  classically for the operator $\mathcal{A}_{p}$.
\end{proof}
\begin{proof}[Proof of Theorem~\ref{thm:easy}]
  The proof is based on the construction of an approximate sequence $(\bfv^{m})\subset
  W^{1,3}_{0,\sigma}(\Omega)$ which solves the following regularized problem (written with
  the weak formulation)
\begin{equation}
 \label{eq:regularized1802}
 \begin{aligned}
    \int_{\Omega}\frac{1}{m}|\bfD\bfv^{m}|\bfD\bfv^{m}:&\bfD\bfphi +
    \nu_{0}\,\bfD\bfv^{m}:\bfD\bfphi+d^{2}|\bfomega^{m}|\bfomega^{m}\cdot\curl\bfphi
    \\
    &+(\nlm)\cdot\bfphi\,\mathrm{d}\bfx 
    =\langle\bff,\bfphi\rangle,\qquad\forall\,\bfphi\in W^{1,3}_{0,\sigma}(\Omega),
  \end{aligned}
\end{equation}
The regularization is a technical step necessary to have a continuous problem,
approximating~\eqref{eq:BLM} and for which the difference $\bfv^{m}-\bfv$ can be localized
to produce a legitimate test function (This is not easy to be done at the finite
dimensional level).

The construction of the solution $\bfv^{m}$ goes through a Galerkin approximation
$\bfv^{m}_{n}\in V_{n}$,
\begin{equation*}
  \begin{aligned}
    \frac{1}{m}\int_{\Omega}|\bfD\bfv^{m}_{n}|\bfD\bfv^{m}_{n}:\bfD\bfphi_{j} +
    \nu_{0}\,\bfD\bfv^{m}_{n}:\bfD\bfphi_{j}+d^{2}|\bfomega^{m}_{n}|\bfomega^{m}_{n}\cdot\curl\bfphi_{j}
    \\
    +(\nlmn)\cdot\bfphi_{j}\,\mathrm{d}\bfx =\langle\bff,\bfphi_{j}\rangle\qquad \text{for }j=1,\dots,n,
  \end{aligned}
\end{equation*}
where $V_{n}=\text{Span}\{\bfphi_1,\dots,\bfphi_n\}$ and
$\bfomega^{m}_{n}=\curl\bfv^{m}_{n}$.  The functions $(\bfphi_i)_{i}$ are a Galerkin basis made by
smooth and solenoidal functions. 

Using $\bfv^{m}_{n}$ as test function gives the estimate
\begin{equation*}
\int_{\Omega}  \frac{1}{m}|\bfD\bfv^{m}_{n}|^{3}+
  \frac{\nu_{0}}{2}|\bfD\bfv^{m}_{n}|^{2}+
  d^{2}|\bfomega^{m}_{n}|^{3}\,\mathrm{d}\bfx 
  \leq\frac{C_{K}^{2}}{2\nu_0}\|\bff\|_{-1,2}^2,
\end{equation*}
where $C_K$ is the constant in Korn's inequality~\eqref{eq:Korn}.
%the a-priori estimates
%\begin{equation*}
%  \int_{\Omega}|\nabla\bfv^{m}|^{2}\,dx\leq C\qquad   \int_{\Omega}d^{2}|\bfomega^{m}|^{3}\,dx\leq C,
%\end{equation*}
%where we also used Korn inequality.
Hence, using Korn inequality, we have (up to a sub-sequence) that for fixed $m \in \N$
\begin{align}
  & \bfv^{m}_{n}\overset{n}{\rightharpoonup} \bfv^{m}\qquad\text{in
  }W^{1,3}_{0,\sigma}(\Omega),\label{eq:1802A}
  \\
  & \bfv^{m}_{n}\overset{n}{\rightharpoonup} \bfv\qquad\text{in }L^{q}(\Omega),\qquad
  \forall\,q<\infty,\label{eq:1802B}
\end{align}
This regularity is enough to apply the classical monotonicity argument
(cf.~\cite[p.~171,p.~216]{Lio1969}). In particular, from~\eqref{eq:1802A}-\eqref{eq:1802B}
it follows that 
\begin{equation*}
  \begin{aligned}
    & \int_{\Omega}
    (\nlmn)\cdot\bfv^{m}_{n}\,\mathrm{d}\bfx\overset{n}{\to}\int_{\Omega}(\nlm)\cdot\bfv^{m}
    \,\mathrm{d}\bfx,
 %   \\
 %   &\int_{\Omega}d^{2}|\bfomega^{m}|\bfomega^{m}\cdot\bfpsi\,\mathrm{d}\bfx=\int_{\Omega}d^{4/3}|\bfomega^{m}|\bfomega^{m}\cdot
 %   d^{2/3}\bfpsi\,\mathrm{d}\bfx
 %   \\
 %   &\hspace{4cm} \to\int_{\Omega}\chi \cdot d^{2/3}\bfpsi\,\mathrm{d}\bfx=\int_{\Omega}d^{2/3}\chi \cdot \bfpsi\,\mathrm{d}\bfx,
  \end{aligned} 
\end{equation*}
Next, the function $\bfv^{m}\in W^{1,3}_{0,\sigma}(\Omega)$ is a weak solution in the
sense of~\eqref{eq:regularized1802}. This can be proved by observing that if we define the
following operator
\begin{equation*}
\mathcal{B}^{1/m}(\bfw):=-\frac{1}{m}\,\Div|\bfD\bfw|\bfD\bfw-
\nu_{0}\,\Div\bfD\bfw+\curl(d^{2}\,|\curl\bfw|\curl\bfw), 
\end{equation*}
 it holds that
\begin{equation*}
  0\leq  \int_\Omega
  \big(\mathcal{B}^{1/m}(\bfv^m_{n})-\mathcal{B}^{1/m}(\bfw)\big):(\bfv^m_{n}-\bfw)\,\mathrm{d}\bfx\qquad\forall
  \,\bfw\in   W^{1,3}_{0,\sigma}(\Omega),
\end{equation*}
(the later inequality holds not only formally, but rigorously, since integral is
well-defined). Moreover, being $\bfv^{m}_{n}$ a legitimate test function in the Galerkin
formulation, it is possible to pass to the limit (for fixed $m\in\N$) as $n\to\infty$,
showing that (exactly as in~\cite{Lio1969}, where  the tools for generalized Navier-Stokes
equations have been developed) 
\begin{equation*}
0\leq  \int_\Omega
\big(\mathcal{B}^{1/m}(\bfv^{m})-\mathcal{B}^{1/m}(\bfw)\big):(\bfv^{m}-\bfw)\,\mathrm{d}\bfx\qquad\forall\,
\bfw\in W^{1,3}_{0,\sigma}(\Omega).
\end{equation*}
Choosing $\bfw=\bfv^{m}-\lambda\bfphi$, with $\lambda>0$ and arbitrary $\bfphi\in
W^{1,3}_{0,\sigma}(\Omega)$, this is enough to infer that
$\lim_{n\to+\infty}\mathcal{B}^{1/m}(\bfv^{m}_{n})=\mathcal{B}^{1/m}(\bfv^{m})$.

\bigskip

To study the limit $m\to+\infty$ for the sequence $(\bfv^{m})$ a technique beyond the
classical monotonicity is needed.

First, taking $\bfv^{m}$ as test function in~\eqref{eq:regularized1802} we get 
\begin{equation}
  \label{eq:1902}
  \begin{aligned}
    \int_{\Omega}\frac{1}{m}|\bfD\bfv^{m}|^{3} +
    \frac{\nu_{0}}{2}\,|\bfD\bfv^{m}|^{2}+d^{2}|\bfomega^{m}|^{3}\,\mathrm{d}\bfx\leq
    \frac{C_{K}^{2}}{2\nu_{0}}\|\bff\|_{-1,2}^{2}.
  \end{aligned}
\end{equation}
%
%Passing to the limit in the energy equality we get 
%\begin{equation}
%  \label{eq:limit1}
%  \int_{\Omega}\nu_{0}|\bfD\bfv^{m}|^{2}+d^{2}|\bfomega^{m}|^{3}\,\mathrm{d}\bfx
%  =\langle\bff,\bfv^{m}\rangle\to \langle\bff,\bfv\rangle,
%\end{equation}
%while passing to the limit in the weak formulation, at $j\in\N$ fixed, and then arguing by
%density, we have
%\begin{equation*}
%    \int_{\Omega}\nu_{0}\,\bfD\bfv:\bfD\bfphi+d^{2/3}\chi\cdot\curl\bfphi+(\nl)\cdot\bfphi\,\mathrm{d}\bfx
%  =\langle\bff,\bfphi\rangle\qquad\forall
%  \,\bfphi\in C^{\infty}_{0,\sigma}(\Omega). 
%\end{equation*}
%
%
%
%
Hence, using Korn inequality, we have (up to a sub-sequence)
\begin{align}
    &  \frac{1}{m} |\bfD\bfv^{m}|\bfD\bfv^{m}\rightharpoonup \bfzero\qquad\text{in
    }L^{3/2}(\Omega), \label{eq:1302A-1}
    \\
    &   \bfv^{m}\rightharpoonup \bfv\qquad\text{in }W^{1,2}_{0,\sigma}(\Omega),\label{eq:1302A}
    \\
    & \bfv^{m}\rightarrow \bfv\qquad\text{in }L^{q}(\Omega),\qquad \forall\,q<6,\label{eq:1302B}
    \\
    &d^{4/3}|\bfomega^{m}|\bfomega^{m}\rightharpoonup \chi\qquad\text{in
    }L^{3/2}(\Omega). \label{eq:1302C}
\end{align}
This implies in particular that 
\begin{equation*}
  \begin{aligned}
    & \int_{\Omega} (\nlm)\cdot\bfv^{m}\,\mathrm{d}\bfx\to\int_{\Omega}(\nl)\cdot\bfv \,\mathrm{d}\bfx,
    \\
    &\int_{\Omega}d^{2}|\bfomega^{m}|\bfomega^{m}\cdot\bfpsi\,\mathrm{d}\bfx=\int_{\Omega}d^{4/3}|\bfomega^{m}|\bfomega^{m}\cdot
    d^{2/3}\bfpsi\,\mathrm{d}\bfx 
    \\
    &\hspace{4cm} \to\int_{\Omega}\chi \cdot d^{2/3}\bfpsi\,\mathrm{d}\bfx=\int_{\Omega}d^{2/3}\chi \cdot \bfpsi\,\mathrm{d}\bfx,
  \end{aligned} 
\end{equation*}
for all $\bfpsi \in L^{3}(\Omega)$.
%Passing to the limit in the energy equality we get 
%\begin{equation}
%  \label{eq:limit1}
%  \int_{\Omega}\nu_{0}|\bfD\bfv^{m}|^{2}+d^{2}|\bfomega^{m}|^{3}\,\mathrm{d}\bfx
%  =\langle\bff,\bfv^{m}\rangle\to \langle\bff,\bfv\rangle,
%\end{equation}
%while 
Passing to the limit in the weak formulation  we have
\begin{equation*}
    \int_{\Omega}\nu_{0}\,\bfD\bfv:\bfD\bfphi+d^{2/3}\chi\cdot\curl\bfphi+(\nl)\cdot\bfphi\,\mathrm{d}\bfx
  =\langle\bff,\bfphi\rangle\qquad\forall
  \,\bfphi\in C^{\infty}_{0,\sigma}(\Omega). 
\end{equation*}
%Hence, we can also set $\bfphi=\bfv$ to obtain
%\begin{equation}
%  \label{eq:limit2}
%    \int_{\Omega}\nu_{0}|\bfD\bfv|^{2}+d^{2/3}\chi\cdot\bfomega \,\mathrm{d}\bfx =\int_{\Omega}\bff\cdot\bfv.
%\end{equation}
If we formally rewrite now the inequality
\begin{equation*}
%  \label{eq:3}
  0\leq\int_{\Omega}(d^2|\bfomega^m|\bfomega^m-d^2|\bfomega|\bfomega)
  \cdot(\bfomega^m-\bfomega)\,\mathrm{d}\bfx,
\end{equation*}
coming from the monotonicity and express the same quantity by means of the weak
formulation, we can observe that the classical monotonicity argument will work since the
convergence of the generally troubling term
\begin{equation*}
  \int_\Omega\nlm\cdot(\bfv^m-\bfv)\,\mathrm{d}\bfx\to0,
\end{equation*}
trivially follows from the uniform  bound $\|\bfv^m\|_{W^{1,2}}\leq C$.

\bigskip

The crucial point is now that the integral
\begin{equation*}
\int_\Omega  d^2|\bfomega|\bfomega
  \cdot(\bfomega^m-\bfomega)\,\mathrm{d}\bfx,
\end{equation*}
is not defined. In fact, for $\bfv\in W^{1,2}_{0,\sigma}(\Omega)$ we only have
$d^2|\bfomega|\bfomega\in L^1(\Omega)$ and also $\bfomega^m-\bfomega\in L^2(\Omega)$.
%, that is if $\bfv^m-\bfv$ would be in
%$W^{1,\infty}_{0,\sigma}(\Omega)$ to be used as legitimate test functions.
To overcome this problem we observe that for each compact set $K\Subset\Omega$
\begin{equation*}
  \Big(\min_{\bfx\in K} d(\bfx)^2\Big)\int_{K}|\bfomega^m|^3\,\mathrm{d}\bfx\leq\int_\Omega
  d^2|\bfomega^m|^3\,\mathrm{d}\bfx\leq \frac{C_{K}^{2}}{2\nu_{0}}\|\bff\|_{-1,2}^{2}, 
\end{equation*}
hence a completely local argument may work, being $\bfomega$ in $L^{3}(K)$. 
\begin{remark}
  Since the function $(d(\bfx))^2$ is not in the Muckenhoupt ${A}_3$ class, we cannot
  recover global bounds on the sequence $(\nabla\bfv^{m})$ from the a priori
  estimate~\eqref{eq:1902} and Lemma~\ref{lem:positivity}. This is a mathematical
  peculiarity of the Baldwin-Lomax stress tensor.
\end{remark}
To use a local argument we consider the family of compact sets
\begin{equation*}
  K_n:=\Big\{\bfx\in \Omega: \ d(\bfx)\geq\frac{1}{n}\Big\}
  \Subset \Omega,
\end{equation*}
which are nested and invading, that is $K_{n}\subset K_{n+1}$ and $\cup_{n\in\N}
K_n=\Omega$. Hence, by a diagonal argument, up to a further sub-sequence, we can write
that for each $K\Subset \Omega$
\begin{equation*}
  \bfomega^m\rightharpoonup\bfomega \qquad\text{in }L^3(K),
\end{equation*}
where $\bfomega=\curl \bfv$, by uniqueness of the weak limit.

Next, we fix an open ball $B\Subset\Omega$ such that $\overline{2B}\Subset \Omega$ and
localize with a bump function $\eta \in C^\infty_0(2B)$ such that
 \begin{equation}
   \label{eq:bump}
   \chi_B(\bfx) \leq \eta(\bfx) \leq \chi_{2B}(\bfx),
 \end{equation}
and $\abs{\nabla \eta} \leq
c\,R^{-1}$, where $R>0$ is the radius of~$B$. We define the following 
divergence-free function with support in $\overline{2B}$:
\begin{align*}
  \bfw^{m} := \eta \,(\bfv^m - \bfv) - \Bog_{2B}( \nabla \eta \cdot
  (\bfv^m - \bfv)),
\end{align*}
where $\Bog_{2B}$ is the \Bogovskii{} operator on $2B$, acting from $L^p_0(2B)$ to
$W^{1,p}_0(2B)$, cf.~Theorem~\ref{thm:bog}. Since $\nabla \eta \cdot (\bfv^m - \bfv)$ is bounded in $L^6_0(2B)$
by~\eqref{eq:1302A}, we have that $\bfw^m$ is bounded in
$W^{1,6}_{0,\sigma}(2B)$. Moreover, $\bfv^m \to \bfv$ in $L^3(\Omega)$ and the continuity
of the \Bogovskii{} operator $\Bog_{2B}$ implies 
  \begin{align}
    \label{conv:1m}
    \bfw^m&\rightarrow 0\quad \text{in}\quad L^3(2B),
    \\
    \bfw^m&\rightharpoonup 0\quad \text{in}\quad W^{1,3}(2B),\label{conv:2m}
    \\
    \Bog_{2B}( \nabla \eta \cdot (\bfv^m - \bfv))&\rightarrow 0\quad\text{in}\quad
    W^{1,3}_0(2B).\label{conv:3m}
  \end{align}
  The functions $\bfw^m\in W^{1,3}_0(2B)$ and their extensions by zero on
  $\Omega\backslash 2B$ (which still belong to $W^{1,3}_{0}(\Omega)$ and which we denote
  by a slight abuse of notation with the same symbol) are then legitimate test functions,
  since $|\bfomega^m|\bfomega^m$ and $|\bfomega|\bfomega$ both belong to
  $L^{3/2}_{loc}(\Omega)$.

  We then obtain from the weak formulation~\eqref{eq:regularized1802} of the regularized
  problem the following equality
%\marginpar{$\bfv$ is not a test-function for the Galerkin problem!}
  \begin{align*}
    \int_{\Omega}&\eta \big(d^2|\bfomega^m|\bfomega^m-d^2|\bfomega|\bfomega\big)\cdot
    \big(\bfomega^m-\bfomega\big)\,\mathrm{d}\bfx
    \\
    &= -\int_{\Omega} \big(d^2|\bfomega^m|\bfomega^m-d^2|\bfomega|\bfomega\big)\cdot
    \nabla\eta\times\big(\bfv^m-\bfv\big)\,\mathrm{d}\bfx
    \\
    &+\int_{\Omega} \big(d^2|\bfomega^m|\bfomega^m-d^2|\bfomega|\bfomega\big)\cdot
    \curl\Big[\Bog_{2B}( \nabla \eta \cdot (\bfv^m - \bfv))\Big]\,\mathrm{d}\bfx
    \\
    &- \nu_0\int_\Omega \bfD(\bfv^m-\bfv):\bfD\bfw^m\,\mathrm{d}\bfx+
    \int_\Omega\big(%(\nabla\bfv)\bfv- (\nabla\bfve)\bfve
    \nl-\nlm\big)\cdot\bfw^m\,\mathrm{d}\bfx
    \\
    &+\int_\Omega\big( d^{2/3}\chi-d^{2}|\bfomega|\bfomega\big)\cdot
    \curl\bfw^m\,\mathrm{d}\bfx-\frac{1}{m}\int_\Omega |\bfD\bfv^m|\bfD\bfv^m:\bfD\bfw^m\,\mathrm{d}\bfx
    \\
    &=:(I)+(II)+(III)+(IV)+(V)+(VI).
  \end{align*}
  Due to the strong $L^3$ convergence of $\bfv^m$ and~\eqref{conv:3m} we see that $(I)$
  and $(II)$ vanish as $m\to+\infty$ (We also used that the function $d$ is uniformly
  bounded).
%\marginpar{I changed the argument for $(III)$} 
  We write the following equality 
  \begin{align*}
    (III)&= -\nu_0\int_\Omega \eta|\bfD(\bfv^m-\bfv)|^2-\nu_0\int_\Omega
    \bfD(\bfv^m-\bfv):\nabla\eta\otimes(\bfv^m-\bfv)\,\mathrm{d}\bfx
    \\
    &+\nu_0\int_\Omega \bfD(\bfv^m-\bfv):\bfD\big[\Bog_{2B}(
    \nabla \eta \cdot (\bfv^m - \bfv))\big]\,\mathrm{d}\bfx,
  \end{align*}
  where the first term is non-positive and the second and third one vanish on account
  of~\eqref{eq:1302B} and~\eqref{conv:3m} respectively.  The convergence of $(IV)$ follows
  trivially from the uniform bounds in $W^{1,2}(\Omega)$ and~\eqref{conv:1m}.  The term
  $(V)\rightarrow0$ due to~\eqref{conv:3m} and the bound in $L^{3/2}(B)$ of $\chi$ and
  $|\bfomega|\bfomega$. Finally, $(VI)\to0$, due the $W^{1,3}(B)$ bound of $\bfv^m-\bfv$
  and~\eqref{eq:1302A-1} .

In conclusion, since $\eta\geq0$, the integrand is non-negative by
Lemma~\ref{lem:positivity}, and from  $\eta\equiv1$ on $B$, it follows 
\begin{equation*}
  \begin{aligned}
    0&\leq \int_{B} \big(d^2\,|\bfomega^m|\bfomega^m-d^2|\bfomega|\bfomega\big)\cdot
    \big(\bfomega^m-\bfomega\big)\,\mathrm{d}\bfx
    \\
    &\quad \leq \int_{\Omega}\eta
    \big(d^2\,|\bfomega^m|\bfomega^m-d^2|\bfomega|\bfomega\big)\cdot
    \big(\bfomega^m-\bfomega\big)\,\mathrm{d}\bfx.
  \end{aligned}
\end{equation*}
Consequently, we obtain
  \begin{align*}
    \lim_{m\rightarrow\infty}\int_{B}&
    \big(d^2\,|\bfomega^m|\bfomega^m-d^2|\bfomega|\bfomega\big)\cdot
    \big(\bfomega^m-\bfomega\big)\,\mathrm{d}\bfx=0,
  \end{align*}
and so,  
  \begin{align*}
    d^{2/3}\bfomega^m&\rightarrow d^{2/3}\bfomega\quad \text{a.e in }B.
  \end{align*}
Finally, we use $d(B,\partial\Omega)>R$ and that the distance $d(\bfx)$ is strictly positive for
each $\bfx\in \Omega$.  The arbitrariness
of $B$ implies
  \begin{align*}
    \bfomega^m&\rightarrow \bfomega\quad \text{a.e in }\Omega.
  \end{align*}
  Next, the limit function $\bfomega$ belongs to $L^2(\Omega)$ and it is finite almost
  everywhere. The hypotheses of Vitali's convergence theorem are satisfied since
\begin{equation*}
  \begin{aligned}
&     d^{4/3}|\bfomega^m|\bfomega^m\qquad \text{uniformly bounded in }L^{3/2}(\Omega),
    \\
  &   d^{4/3}|\bfomega^m|\bfomega^m\to     d^{4/3}|\bfomega|\bfomega\quad\text{a.e. in
  }\Omega  ,
  \\
&  \bfomega \text{ finite a.e.},
  \end{aligned}
\end{equation*}
ensuring that $d^{2/3}\chi=d^2|\bfomega|\bfomega$ such that the limit $\bfv$ is a weak
solution to~\eqref{eq:BLM}.
\end{proof}
\section{On generalised Baldwin-Lomax models}
\label{sec:Generalised-BL}
In the proof of the result from the previous section it was essential to have $\nu_0$
positive and fixed, to derive a uniform bound of $(\bfv^{m})_{m\in\N}$ in
$W^{1,2}_0(\Omega)$. This allows us to make sense of the boundary conditions, among the
other relevant properties.  On the other hand, in applications $\nu_{0}$ is generally an
extremely small number. The K41-Kolmogorov theory for turbulence is in fact valid in the
vanishing viscosity limit, and predicts (still in a statistical sense) a non zero
turbulent dissipation, see Frisch~\cite{Fri1995}.  To capture the properties which are
still valid in the limit $\nu_0=0$ we study now the following steady system
\begin{align*}
  \left\{\begin{aligned} 
      (\nabla\bfv)\bfv+\curl\big(d^\alpha(\kappa+|\bfomega|)^{p-2}\bfomega\big)
      +\nabla \pi&=\bff& \mbox{in        $\Omega$,}
      \\
      \Div \bfv&=0\quad& \mbox{in $\Omega$,}
      \\
      \bfv&=\mathbf{0}\quad\quad& \mbox{ \,on $\partial\Omega$.}
\end{aligned}\right.\label{eq:1.1}
\end{align*}
Here $\kappa\geq0$ and the most \textit{interesting} case is the following (degenerate)
one
\begin{equation*}
  \kappa=0,\qquad p=3,\qquad \alpha=1,
\end{equation*}
where the exponent $p=3$ is exactly that from the turbulence theory (as a generalization
of the classical Smagorinsky theory), while $\alpha=1$ is the same as suggested
in~\eqref{eq:Kelvin-Voigt} from the model introduced in~\cite{ABLN2019}.
%  with 
%\begin{equation*}
%  \nu_{T}(\bfv(\bfx))=\ell_0 \ell(\bfx)|\curl\bfv(\bfx)|,\qquad \text{for some }\ell_0>0.
%\end{equation*}
Without loss of generality we also set $\ell_0=1$ and $\ell(\bfx)=d(\bfx)$, as in the
turbulent viscosity described in Remark~\ref{rem:alpha1}.
\begin{remark}
  The critical value (coming from the Muckenhoupt theory, cf.~Section~\ref{sec:weighted})
  for the power of the distance is $\alpha=p-1=3-1=2$. In this case certain bounds on the
  first derivatives of the velocity can be still inferred  from weighted estimates of the
  gradient, as in~\eqref{eq:grad-curl-weighted}.
\end{remark}

\bigskip 

\noindent We start our analysis focusing on the following boundary value problem still
written in rotational form
\begin{align}
\label{eq:BL-steady}
  \left\{\begin{aligned}
\nl%    (\nabla\bfv)\bfv
+\curl \big(d\,|\bfomega|\bfomega\big)+\nabla \pi&=\bff& \mbox{in
      $\Omega$,}
    \\
    \Div \bfv&=0\qquad& \mbox{in $\Omega$,}
    \\
    \bfv&=\mathbf{0}\,\qquad\quad& \mbox{ \,on $\partial\Omega$.}
\end{aligned}\right.
\end{align}
\begin{definition}
\label{def:weak-solution-alpha1-p3}
  We say that $\bfv\in W^{1,3}_{0,\sigma}(\Omega,d)$ is a weak solution
  to~\eqref{eq:BL-steady} if the following equality is satisfied
  \begin{equation*}
    \int_\Omega(\nl)%(\nabla\bfv)\bfv
\cdot\bfphi+%\int_\Omega
    d\,|\bfomega|\bfomega\cdot\curl\bfphi\,\mathrm{d}\bfx=\int_\Omega\bff
    \cdot\bfphi\,\mathrm{d}\bfx\qquad\forall\,\bfphi\in C^{\infty}_{0,\sigma}(\Omega).
  \end{equation*}
\end{definition}

The main result we will prove in this section is the following.
\begin{theorem}
  \label{thm:existence}
  Let be given $\bff%\in W^{4/3,3/2}(\Omega)$,  such that $\bff
=\Div \bfF$ with $\bfF\in L^{3/2}(\Omega,d^{-1/2})$ then there exists a weak solution $\bfv\in
  W^{1,3}_{0,\sigma}(\Omega,d)$ of the problem~\eqref{eq:BL-steady}.
%, where
%  \begin{equation*}
%    W^{1,3}_{0,\sigma}(d,\Omega):=\overline{\left\{u\in
%        C^\infty_{0,\sigma}(\Omega;\R^3)\right\}}^{\left(\int_\Omega d\,|\nabla
%        u|^3\,dx\right)^{1/3}} .  
%  \end{equation*}
In addition, the solution satisfies the energy-type equality
\begin{equation*}
  \int_\Omega    d\,|\bfomega|^3\,\mathrm{d}\bfx=-\int_\Omega\bfF
    \cdot\nabla\bfv \,\mathrm{d}\bfx.
\end{equation*}
\end{theorem}
\begin{remark}
  By using fractional spaces we have that the same theorem holds for instance if
  $\partial\Omega$ is of class $C^2$ and if
  \begin{equation*}
    \bff\in \widehat{W}^{-2/3,3/2}(\Omega):=(W^{2/3,3}(\Omega)\cap L^3_0(\Omega))'.
\end{equation*}
In fact, by using Thm.~3.4 from Gei\ss{}ert, Heck, and Hieber~\cite{GHH2006} there exists
a bounded linear operator $R: \widehat{W}^{-2/3,3/2}(\Omega)\to W^{1/3,3/2}(\Omega)$, such
that $\Div R(\bff)=\bff$. Next, observe that $W^{1/3,3/2}(\Omega)=W^{1/3,3/2}_0(\Omega)$,
and consequently it follows for $\bfphi\in C^\infty_0(\Omega)$
\begin{equation*}
  \langle\bff,\bfphi\rangle=  \langle\Div R(\bff),\bfphi\rangle=-  \langle
  R(\bff),\nabla\bfphi\rangle=-  \langle  d^{-1/3} R(\bff),d^{1/3}\nabla\bfphi\rangle,
\end{equation*}
and  --with the characterization of fractional spaces
from~\eqref{eq:characterization-fractional}--
\begin{equation*}
\Big|  \langle\bff,\bfphi\rangle\Big|\leq
c\|R(\bff)\|_{1/3,3/2}\|\nabla\bfphi\|_{3,d}\leq\|\bff\|_{-2/3,3/2}\|\nabla\bfphi\|_{3,d}.
\end{equation*}
Then, the estimates follow in the same manner as before.
\end{remark}

\bigskip

Due to the fact that we have a problem without a principal part of standard $p$-Stokes
type, we need to properly approximate~\eqref{eq:BL-steady} in order to construct weak
solutions. As in the previous section we consider, for $\epsilon>0$, the following
approximate system
\begin{align}
  \left\{
    \begin{aligned}
      -\epsilon\,\Div(|\bfD\bfve|\bfD\bfve)+ \nle%(\nabla\bfve)\bfve
      +\curl \big(d\,|\bfoe|\bfoe\big)+\nabla \pi_\epsilon&=\bff& \mbox{in $\Omega$,}
      \\
      \Div \bfve&=0& \mbox{in $\Omega$,}
      \\
      \bfve&=\mathbf{0}& \mbox{on $\partial\Omega$,}
    \end{aligned}\right.\label{eq:approximation}
\end{align}
which falls within in the classical setting as studied starting with the work of
Lady\v{z}henskaya~\cite{Lad1969} and Lions~\cite{Lio1969}.
\begin{remark}
  At this stage (existence of weak solutions for the approximate problem) the power of
  $d(\bfx)$ entering in the equations does not play any specific role.
\end{remark}
%
%\begin{remark}
%\end{remark}
%
%For this system we can use the
%standard monotonicity argument, which allows to prove existence of weak solutions, since
%$\bfv^{m}$ is regular enough to be used as test function.  We have
With the same tools already used, we have the following result.
\begin{theorem}
  \label{thm:existence-approximation}
  For any $\epsilon>0$ and for $\bff=\Div\bfF$ with $\bfF\in L^{3/2}(\Omega)$ there
  exists a weak solution $\bfve\in W^{1,3}_{0,\sigma}(\Omega)$ which satisfies
  \begin{equation}
    \label{eq:weak-formulation}
    \int_\Omega
    \epsilon |\bfD\bfve|\bfD\bfve:\bfD\bfphi%\int_\Omega(\nabla\bfve)\bfve\cdot\bfphi+\int_\Omega
    +(\nle)\cdot\bfphi+ d\,|\bfoe|\bfoe\cdot\curl\bfphi\,\mathrm{d}\bfx=-\int_\Omega\bfF \cdot\nabla\bfphi\,\mathrm{d}\bfx,
  \end{equation}
  for all $\bfphi\in W^{1,3}_{0,\sigma}(\Omega)$. %C^{\infty}_{0,\sigma}(\Omega)$.
  The function $\bfve$ satisfies the energy-type estimate
  \begin{equation}
    \label{eq:apriori-epsilon1}
    \epsilon\|\bfve\|_{W^{1,3}_{0}}^3+\int_\Omega d\,|\bfoe|^3\leq
    \frac{C}{\sqrt{\epsilon}}    \|\bfF\|_{{3/2}}^{3/2}. 
  \end{equation}
  Moreover, if $\bfF\in  L^{3/2}(\Omega,d^{-1/2})$,
  % $\int_\Omega\frac{|\bfF|^{3/2}}{d^{1/2}}<\infty$, 
then
  \begin{equation}
    \label{eq:apriori-epsilon2}
    \epsilon\|\bfve\|_{W^{1,3}}^3+\int_\Omega d\,|\bfoe|^3\leq
    C\int_\Omega\frac{|\bfF|^{3/2}}{d^{1/2}}\,\mathrm{d}\bfx=C\|\bfF\|_{3/2,d^{-1/2}}^{3/2}, 
\end{equation}
for some constant $C$ independent of $\epsilon$. 
%Here\footnote{Other choices are possible due to the embedding of weighted spaces} we
%consider $X=L^{3/2}(\Omega,d^{-1/2})$ 
\end{theorem}
\begin{remark}
  The approximation in~\eqref{eq:approximation} is introduced only as a mathematical tool,
  no modeling is hidden inside the choice for the perturbation.

  The regularization can be also done in the following way, respecting the rotational
  structure of the equation:
  \begin{equation*}
    \nle+\curl
    \big((\epsilon+d)|\bfoe|\bfoe\big)+\nabla \pi_\epsilon=\bff\qquad \mbox{in $\Omega$.}
  \end{equation*}
  For this approximation one can use the fact that $d+\epsilon\geq \epsilon>0$ and
  $\|\bfoe\|_{p}\sim \|\nabla\bfv\|_{p}$ for functions which are divergence-free and zero
  at the boundary by~\eqref{eq:div-curl}. We preferred the more classical way in order to
  use directly known results, being completely equivalent in terms of existence theorems.
\end{remark}
\begin{proof}[Proof of Theorem~\ref{thm:existence-approximation}]
We do not give the easy proof of this result we just show the basic a priori estimates. 
%\marginpar{I would skip this proof}
%  We look for a Galerkin approximate solution $\bfve^{m}\in V_m$ such that
%\begin{equation*}
%  \begin{aligned}
%    \int_\Omega \epsilon\,|\bfD\bfve^m|\bfD\bfve^m):\bfphi_j+
%    (\nlem)\cdot\bfphi_j%(\nabla\bfve)\bfve+\curl
%    +d\,|\bfoe^m|\bfoe^m\cdot \curl\bfphi_j\,\mathrm{d}\bfx
%    \\
%    =-\int_\Omega\bfF:\nabla\bfphi_j\,\mathrm{d}\bfx
%  \end{aligned}
%\end{equation*}
%for $j=1,\dots,m$, where $V_m=\text{Span}\{\bfphi_1,\dots,\bfphi_m\}$ is the Galerkin space. 
%%
The first $\epsilon$-dependent estimate~\eqref{eq:apriori-epsilon1} is  obtained by using as test function
$\bfve$ itself, integrating by parts, and using H\"older inequality to estimate the
right-hand side.
%  \begin{equaation*}
%        \epsilon\|\bfve\|_{W^{1,3}_{0}}^3+\int_\Omega d\,|\bfoe^m|^3\leq C_\epsilon
%    \|\bfF\|_{L^{3/2}}^{3/2}.
%  \end{equation*}
%%  This ensures that for fixed $\epsilon>0$ the sequence $\bfve^m\in
%%  W^{1,3}_{0,\sigma}(\Omega)$, with bounds independent of $m$. Hence it is possible to
%%  pass to the limit as $m\rightarrow+\infty$ (at fixed $\epsilon>0$).  Observe that the
%%  bound in $W^{1,3}_{0,\sigma}(\Omega)$ implies that we can find $\bfve$ such that (up to
%%  a sub-sequence)
%%  \begin{equation*}
%%    \begin{aligned}
%%      &|\bfD\bfv^m|\bfD\bfv^m+d\,|\bfoe^m|\bfoe^m\rightharpoonup\chi_1+\chi_2&\text{ in }(W^{1,3}_{0,\sigma}(\Omega))',
%%      \\
%%     & \bfve^m\rightharpoonup\bfve&\text{in }W^{1,3}_{0,\sigma}(\Omega),
%%      \\
%%      &\bfve^m\rightarrow\bfve&\text{in }L^{q}(\Omega),\ \forall\,q<\infty.
%%    \end{aligned}
%%  \end{equation*}
%%Hence 
%%\begin{equation*}
%%\int_\Omega  (\nlem)\cdot\bfve^m\,\mathrm{d}\bfx\to\int_\Omega  (\nle)\cdot\bfve\,\mathrm{d}\bfx,
%%\end{equation*}
%%
%%\begin{remark}
%%In the argument above the role of $p=3$ is not   essential and the same proof applies as
%%long as the term $\nlem\cdot\bfve^m$ converges to $\nle\cdot\bfve$ in $L^1(\Omega)$. This
%%works for all $p>\frac{9}{5}$ due to Sobolev's embedding theorem.
%%\end{remark}
%%

In the following we also need estimates which are independent of $\epsilon>0$ and choosing
again $\bfphi=\bfve$ in~\eqref{eq:weak-formulation} the right-hand side can be estimated
by
%
%In our case we have then   $\|\nabla u\|_{L^{3}_{d}}\leq C(d,\Omega)(    \|\nabla\times
%u\|_{L^{3}_{d}})$ or  more explicitly
%\begin{equation}
%\label{eq:weighted-curl}
%\int_\Omega d\,|\nabla u|^3\leq C \int_\Omega 
%  d\,|\nabla\times u|^3\qquad \forall \,u\in W^{1,3}_{0,\sigma}(d,\Omega)
%\end{equation}
%hence the a-priori estimate for system~\eqref{eq:approximation} becomes
  \begin{equation*}
%    \epsilon\|\bfve\|_{W^{1,3}}^3+\int_\Omega d\,|\bfoe|^3 =
\int_\Omega d^{-1/2} \bfF\cdot
    d^{1/3}\nabla\bfve\,\mathrm{d}\bfx\leq
    C\left(\int_\Omega\frac{|\bfF|^{3/2}}{d^{1/2}}\,\mathrm{d}\bfx\right)^{2/3}\left(\int_\Omega
      d\,|\nabla\bfve|^3\,\mathrm{d}\bfx\right)^{1/3}, 
\end{equation*}
using H\"older's inequality.
On account of~\eqref{eq:grad-curl-weighted} and Young's inequality we obtain further
\begin{equation}
    \label{eq:apriori-epsilon2bis}
    \epsilon\|\bfve\|_{W^{1,3}}^3+\int_\Omega d\,|\bfoe|^3+d\,|\nabla\bfve|^3 \,\mathrm{d}\bfx\leq
    C\int_\Omega\frac{|\bfF|^{3/2}}{d^{1/2}}\,\mathrm{d}\bfx,%=C\|\bfF\|_{3/2,d^{-1/2}}^{3/2},
\end{equation}
hence~\eqref{eq:apriori-epsilon2} with a constant $C$ independent of $\epsilon$.

Finally, for $q<3/2$ we have by H\"older's inequality
\begin{equation*}
  \begin{aligned}
    \int_\Omega|\nabla \bfve|^q \,\mathrm{d}\bfx&= \int_\Omega d^{-q/3}\,d^{q/3}|\nabla
    \bfve|^q \,\mathrm{d}\bfx
    \\
    &\leq\left(\int_\Omega d^{-\frac{q}{{3-q}}}\,\mathrm{d}\bfx\right)^{(3-q)/3} \left(\int_\Omega
      d\,|\nabla \bfve|^3 \,\mathrm{d}\bfx\right)^{q/3}
    \\
&    \leq c \left(\int_\Omega d\,|\nabla \bfve|^3 \,\mathrm{d}\bfx\right)^{q/3},
  \end{aligned}
\end{equation*}
such that
\begin{equation*}
 \left(   \int_\Omega|\nabla \bfve|^q\,\mathrm{d}\bfx\right)^{3/q} \leq c \int_\Omega
 d\,|\bfoe|^3\,\mathrm{d}\bfx\leq C\|\bfF\|_{3/2,d^{-1/2}}^{3/2},
\end{equation*}
using~\eqref{eq:apriori-epsilon2bis}.
This proves then that the solution to~\eqref{eq:approximation} satisfies also the
estimate
\begin{equation}
  \label{eq:estimate-not-epsilon}
 \|\nabla\bfve\|_{L^q}\leq C(q,\Omega,\|\bfF\|_{3/2,d^{-1/2}}). 
\end{equation}
for all $q<\frac{3}{2}$.
\end{proof}
Collecting all estimates we can give now the main existence result for the generalized
Baldwin-Lomax model~\eqref{eq:BL-steady}, passing to the limit as $\epsilon\to0$.
\begin{proof}[Proof of Theorem~\ref{thm:existence}]
  Using the existence result from Theorem~\ref{thm:existence-approximation} we obtain a
  sequence of solutions $(\bfve)\subset W^{1,3}_{0,\sigma}(\Omega)$
  to~\eqref{eq:approximation}. From the uniform
  estimates~\eqref{eq:apriori-epsilon2}-\eqref{eq:estimate-not-epsilon} we infer the
  existence of a limit function $\bfv\in W^{1,q}_{0,\sigma}(\Omega)$ such that along a
  sequence $\epsilon_{m}\to0$ and for $\bfv^{m}:=\bfv_{\epsilon_{m}}$ it holds
  \begin{align}
    \bfvm&\rightharpoonup \bfv\qquad\text{in }W^{1,q}_{0,\sigma}(\Omega)\qquad
    \forall\,q<\frac{3}{2},\label{eq:a}
    \\
    \bfvm&\rightarrow \bfv\qquad\text{in
    }L^{r}_{\sigma}(\Omega)\qquad\forall\,r<3,\label{b}
    \\
    \bfvm&\rightarrow \bfv\qquad \text{a.e. in }\Omega,\label{c}
    \\
    \varepsilon_{m}|\bfD\bfvm|\bfD\bfvm&\rightarrow \mathbf{0}\qquad\text{in
    }L^{3/2}_{0}(\Omega).\label{d}
  \end{align}
  At this point we observe that it is not possible to pass to the limit as $\epsilon\to0$
  in the equations directly by monotonicity arguments since
  $\frac{3}{2}<\frac{9}{5}$. Hence the difficulty will be again proving that $\bfv$ is a
  weak solution to~\eqref{eq:BL-steady}. We will employ a local argument similar to the
  previous section. For all compact sets $K\Subset\Omega$ it holds that
%  \begin{equation*}
%    \exists\,C_K>0\qquad \text{such that}\qquad\min_{x\in K}d(x)\geq C_K.
%  \end{equation*}
%  hence it implies, by using the weighted estimates from Section~\ref{sec:weighted}, that
%
%
  \begin{equation*}
    \begin{aligned}
      c_0 \big(\min_{\bfx\in K}d(\bfx)\big)\int_K|\nabla \bfvm|^3\,\mathrm{d}\bfx&\leq c_0\int_K d\,|\nabla
      \bfvm|^3\,\mathrm{d}\bfx
      \\
      &\leq c_0 \int_\Omega d\,|\nabla \bfvm|^3\,\mathrm{d}\bfx\leq C(\Omega,\|\bfF\|_{3/2,d^{-1/2}}),
    \end{aligned}
  \end{equation*}
  using~\eqref{eq:apriori-epsilon2}. This
  shows that (up to possibly another sub-sequence)
  \begin{equation}
\label{conv:loc}
    \begin{aligned}
      & (\nabla \bfvm)_{|K}\rightharpoonup \nabla\bfv_{|K}\qquad\text{in }L^{3}(K)\qquad
      \forall\,K\Subset\Omega,
      \\
      &(\bfvm)_{|K}\rightarrow \bfv_{|K}\qquad\text{in }L^{r}(K)\qquad\forall\,r<\infty.
    \end{aligned}
  \end{equation}
  This proves that 
  \begin{equation*}
%    \label{eq:weak-formulation-2}
    \int_\Omega (\nlm)\cdot\bfphi\,\mathrm{d}\bfx\xrightarrow{m\to\infty}
    \int_\Omega(\nl)\cdot\bfphi\,\mathrm{d}\bfx\qquad \forall\,\bfphi\in C^{\infty}_{0,\sigma}(\Omega),
  \end{equation*}
  while passing to the limit in the nonlinear term requires again a local  approach, as
  developed in the previous section.

  Based on the previous observations if $\overline\bfS$ denotes the
  $L^{3/2}_{loc}(\Omega)$-weak limit of $d\,|\bfoe|\bfoe$, which exists by using the uniform
  bound coming from~\eqref{eq:apriori-epsilon2}, we obtain the limit system
  \begin{align}
    \left\{\begin{array}{rc} \nl%(\nabla\bfv)\bfv
        +\curl \overline\bfS+\nabla \pi=\Div\bfF&
        \mbox{in $\Omega$,}
        \\
        \Div \bfv=0\qquad& \mbox{in $\Omega$,}
        \\
        \bfv=0\qquad\quad& \mbox{ \,on $\partial\Omega$,}
      \end{array}\right.\label{eq:limit}
  \end{align}
  where the first equation is satisfied in the sense of distributions over $\Omega$. The
  remaining effort is to show that $\overline\bfS=d\,|\bfomega|\bfomega$.

  Observe also that at this point we have that $\nle% (\nabla\bfvm)\bfvm
 \in L^s_{loc}(\Omega)\subset L^1_{loc}(\Omega) $ for all $s<3$, but not uniformly in $\epsilon$.

 The uniform estimates imply that $\bfvm\in W^{1,q}_{0}(\Omega)$, for all $q<3/2$, hence
 $\bfvm\in L^r(\Omega)$, for all $r<3$. This is not enough to show    $ \nlm% (\nabla\bfv)\bfv
 \in L^1(\Omega)$, hence testing with $\bfv$ itself seems not possible. 
%Nevertheless one
% has clearly $\nle\in L^1_{loc}(\Omega)$, being the function $d$ bounded away from zero in
% any $K\Subset\Omega$, hence a local argument may be employed.
%
%

 First, we improve the known summability of the solutions, by observing that
 applying~\eqref{eq:gen-Sobolev} to our case ($p=3$, $\delta=1/3$) implies
  \begin{equation*}
    % \label{eq:gen-Sobolev}
    \Big\|\bfvm(\bfx)-\dashint_\Omega \bfvm(\bfy)\,\mathrm{d}\bfy\Big\|_{9}^3\leq
    C\|d^{1/3} \nabla \bfvm\|_3^3=\int_\Omega d\,|\nabla \bfvm|^3\,\mathrm{d}\bfx\leq
    C\|\bfF\|_{3/2,d^{-1/2}}^{3/2},
  \end{equation*}
  uniformly in $\varepsilon$.
  Next we recall that by H\"older inequality
%  \begin{equation*}
$    \Big\|\dashint_\Omega f\,\mathrm{d}\bfx\Big\|_{p}\leq\|f\|_{p}, $
 % \end{equation*}
  such that
  \begin{equation*}
    \|f\|_p-\Big\|\dashint_\Omega f(\bfy)\,\mathrm{d}\bfy\Big\|_{p}\leq\Big\|f(\bfx)-\dashint_\Omega f(\bfy)\,\mathrm{d}\bfy\Big\|_{p},
  \end{equation*}
  for any $f\in L^p(\Omega)$.
  This yields due to the embedding into $L^r(\Omega)\subset
  L^1(\Omega)$ for $r<3$
  \begin{equation*}
    \begin{aligned}
      \|\bfvm\|_9&\leq \Big\|\dashint_\Omega \bfvm\,\mathrm{d}\bfy\Big\|_{9}+C\|\bfF\|_{3/2,d^{-1/2}}^{1/2}
      \\
      &\leq
      \frac{1}{|\Omega|^{8/9}}\|\bfvm\|_1+C\|\bfF\|_{3/2,d^{-1/2}}^{1/2}\leq c(|\Omega|,\|\bfF\|_{3/2,d^{-1/2}}).
    \end{aligned}
  \end{equation*}
  Finally, we obtain
  \begin{equation*}
    \nlm %(\nabla\bfvm)\bfvm
\in L^s(\Omega)\qquad\forall\,s<\frac{9}{7},
  \end{equation*}
  uniformly in $m\in\N$.  We can also improve~\eqref{b} to
  \begin{equation*}
    \bfvm\rightarrow \bfv\qquad\text{in
    }L^{r}_{\sigma}(\Omega)\qquad\forall\,r<9.%\label{b'}
  \end{equation*}
  Now we consider the difference of~\eqref{eq:approximation} and~\eqref{eq:limit} and
%
%
%
%
%
%
%
%big)+\nabla \pi_\epsilon-\nabla \pi
%
%
%
%g).
%
%
%the term 
%
%bfve-\bfD\bfv)
%
%iformly bounded in
%
%
%.
%
% we obtain
%
%fomega\big)\cdot
%
%fD\bfvm-\bfD\bfv)+\int_\Omega\big((\nabla\bfv)\bfv-
%\&+\int_\Omega\big(
%\bfoe-\bfomega\big)=(I)+(II)+(III).
%
%uniformly bounded in $L^{3/2}$. We
%
%
%
%
%  This problem can be overcome by localising 
localize as in Section~\ref{sec:Classical-BL}, 
  taking into account~\eqref{conv:loc}. Given the bump function as in~\eqref{eq:bump} we
  define
  \begin{align*}
  \bfw^{m} := \eta\, (\bfv^m - \bfv) - \Bog_{2B}( \nabla \eta \cdot
  (\bfv^m - \bfv))\in
 W^{1,3}_{0,\sigma}(2B)\subset W^{1,3}_{0,\sigma}(\Omega),
  \end{align*}
  and we have, due to the $W^{1,3}_{loc}(\Omega)$-bounds from cf.~\eqref{conv:loc}, that
  the same convergence as in~\eqref{conv:1m}-\eqref{conv:2m}-\eqref{conv:3m} holds true.
% \begin{align}
%   \label{conv:1}
%   \bfwe&\rightarrow 0\quad \text{in}\quad L^3(2B),
%   \\
%   \bfwe&\rightharpoonup 0\quad \text{in}\quad W^{1,3}(2B),\label{conv:2}
%   \\
%   \Bog_{2B}( \nabla \eta \cdot (\bfvm - \bfv))&\rightarrow 0\quad\text{in}\quad
%   W^{1,3}(2B).\label{conv:3}
% \end{align}
  Now we test the difference between the $\epsilon_{m}$-regularized system and the original
  one with $\bfw^{m}%the null extension off $2B$ (called again) $\bfwe$
  \in  W^{1,3}_{0,\sigma}(\Omega)$ and by using the same argument as before we get
  % 
  % to
  % obtain
  % \begin{align*}
  %   \int_{\Omega}&\eta \big(d\,|\bfoe|\bfoe-d|\bfomega|\bfomega\big)\cdot
  %   \big(\bfoe-\bfomega\big)
  %   \\
  %   &= -\int_{\Omega} \big(d\,|\bfoe|\bfoe-d|\bfomega|\bfomega\big)\cdot
  %   \nabla\eta\times\big(\bfvm-\bfv\big)
%    \\
%    &+\int_{\Omega} \big(d\,|\bfoe|\bfoe-d|\bfomega|\bfomega\big)\cdot \curl\Bog_{2B}(
%    \nabla \eta \cdot (\bfvm - \bfv))
%    \\
%    &+
%    \epsilon\int_\Omega\,|\bfD\bfvm|\bfD\bfvm:\bfD\bfwe+
%\int_\Omega\big(%(\nabla\bfv)\bfv-    (\nabla\bfvm)\bfvm
%\nl-\nle\big)\cdot\bfwe
%    \\
%    &+\int_\Omega\big( \overline\bfS-d|\bfomega|\bfomega\big)\cdot
%    \curl\bfwe=(I)+(II)+(III)+(IV)+(V).
%  \end{align*}
%By~\eqref{d} and~\eqref{conv:2} we have $(III)\rightarrow0$
%  as $\varepsilon\rightarrow0$, while all other terms converge to zero as can be seen
%  reasoning as in the previous section.
%% we have that all termand due to~\eqref{conv:loc} we get $(I)\rightarrow0$ as $\varepsilon\rightarrow0$.
%%  Combining~\eqref{conv:loc} and~\eqref{conv:3} shows that $(II)\rightarrow0$ as
%%  $\varepsilon\rightarrow0$. As a consequence of~\eqref{conv:loc} and~\eqref{conv:1}
%%  $(IV)\rightarrow0$ as $\varepsilon\rightarrow0$. Finally, $(V)\rightarrow0$ by
%% ~\eqref{conv:2}. 
%In conclusion
  \begin{align*}
    \lim_{m\to+\infty}\int_{B}&
    \big(d\,|\bfon|\bfon-d|\bfomega|\bfomega\big)\cdot \big(\bfon-\bfomega\big)\,\mathrm{d}\bfx=0.
  \end{align*}
This can be used to show that %  Combining this with~\eqref{eq:lowerbound} and using $d\geq \inf_B d>0$ in $B$ shows
  \begin{align*}
    \bfon&\rightarrow \bfomega\quad \text{in}\quad L^3(B),
  \end{align*}
  and since the ball $B\Subset\Omega$ is arbitrary, this implies
  $\overline\bfS=d\,|\bfomega|\bfomega$.

\medskip

%\marginpar{Describe better $\bff$ or $\bfF$}
We finally prove the energy-type balance. We observe that the equality
  \begin{equation*}
    \int_\Omega(\nl)%(\nabla\bfv)\bfv
\cdot\bfphi +%\,\mathrm{d}\bfx+\int_\Omega
    d\,|\bfomega|\bfomega\cdot\curl\bfphi\,\mathrm{d}\bfx=-\int_\Omega\bfF
    \cdot\nabla\bfphi\,\mathrm{d}\bfx,
  \end{equation*}
by density makes sense also for $\bfphi\in W^{1,3}_{0,\sigma}(\Omega,d)$, being the
integrals well defined by the following estimates for $q=\frac{9}{7}<\frac{3}{2}$
\begin{equation*}
\begin{aligned}
  &\left|
    \int_\Omega%(\nabla\bfv)\bfv
(\nl)\cdot\bfphi\,\mathrm{d}\bfx\right|\leq\|\nabla\bfv\|_q\|\bfv\|_9\|\bfphi\|_9\leq 
  c \|\bfv\|^2_{W^{1,3}_0(\Omega,d)}\|\bfphi\|_{W^{1,3}_0(\Omega,d)},
  \\
  &\left|\int_\Omega d\,|\bfomega|\bfomega\cdot\curl\bfphi\,\mathrm{d}\bfx\right|= \left|\int_\Omega
    d^{2/3}\,|\bfomega|\bfomega\cdot d^{1/3}\curl\bfphi\right|\leq c
  \|\bfv\|^2_{W^{1,3}_0(\Omega,d)}\|\bfphi\|_{W^{1,3}_0(\Omega,d)},
  \\
  &\left|\int_\Omega\bfF \cdot\nabla\bfphi\,\mathrm{d}\bfx\right|\leq\|\bfF\|_{3/2,d^{-1/2}}
  \|\bfphi\|_{W^{1,3}_0(\Omega,d)}. 
\end{aligned}
\end{equation*}
Note that we used again~\eqref{eq:gen-Sobolev} with $p=3$ and $\delta=\frac{1}{3}$.
Hence, by setting $\bfphi=\bfv$ and by observing that
\begin{equation*}
  \int_\Omega(\nl)%(\nabla\bfv)\bfv
\cdot\bfv\,\mathrm{d}\bfx=0,
\end{equation*}
once it is well-defined, we get the claimed energy equality.
\end{proof}
\begin{remark}
  Since the convergence is based on local $W^{1,3}$-estimates, the convergence of the
  stress tensor does not depend on the power of the distance, while the range of $\alpha$
  is crucial to handle the convective term and to give a proper meaning to the equations in
  the sense of distributions.
\end{remark}
\section{Extension to more general cases}
In this section we consider the same problem as in~\eqref{eq:BL-steady} but we consider
different values of both the exponent $p$ and of the weight $\alpha$. Some results follow
in a straightforward way since $p=3$ (the main argument of monotonicity requires in fact
$p>\frac{9}{5}$, while others for smaller values of $p$ require a more technical argument
with a Lipschitz truncation of the test functions).
\subsection{Generalization to other values of the parameter $\alpha$, but still with
  $p=3$.} 
We consider now the possible extension to larger values of the parameter $1\leq
\alpha<2$. As explained before the value $\alpha=2=3-1$ is critical as it does not allow
to bound the weighted gradient by the weighted curl. We study now the system
\begin{align}
\label{eq:BL-steady-generalized}
  \left\{\begin{aligned}
\nl%    (\nabla\bfv)\bfv
+\curl \big(d^\alpha|\bfomega|\bfomega\big)+\nabla \pi&=\bff& \mbox{in
      $\Omega$,}
    \\
    \Div \bfv&=0\qquad& \mbox{in $\Omega$,}
    \\
    \bfv&=\mathbf{0}\qquad\quad& \mbox{ \,on $\partial\Omega$.}
\end{aligned}\right.
\end{align}

We write just the a priori estimates, since the approximation and the passage to the limit
is exactly the same as in Theorem~\ref{thm:existence-approximation} being based on local
estimates for the gradient in $L^{3}(K)$.

From the H\"older inequality we get for $1\leq\alpha<2$ and if $\frac{\alpha q}{3-q}<1$
(which is if $1\leq q<\frac{3}{1+\alpha}$) that 
\begin{equation*}
  \begin{aligned}
    \|\nabla \bfv\|_q^3&    \leq c \int_\Omega d^\alpha\,|\nabla\bfv|^3\,\mathrm{d}\bfx \qquad
    \forall\,\bfv\in W^{1,3}(\Omega,d^{\alpha}),
  \end{aligned}
\end{equation*}
Next, the Sobolev embedding from Lemma~\ref{lem:Sobolev} yields 
  \begin{equation*}
%    \label{eq:gen-Sobolev2}
    \Big\|\bfv(\bfx)-\dashint_\Omega \bfv(\bfy)\,\mathrm{d}\bfy\Big\|_{9/\alpha}^3\leq
    C\int_{\Omega}d^\alpha |\nabla \bfv|^3\,\mathrm{d}\bfx\qquad \forall\,\bfv\in W^{1,3}(\Omega,d^{\alpha}).
  \end{equation*}
At this point the convective term satisfies
\begin{equation*}
 (\nabla \bfv)\bfv\in L^{s}(\Omega)\qquad\forall\,s<\frac{9}{3+4\alpha},
\end{equation*}
and $s\geq 1$ if $\alpha<\frac{3}{2}$. Under this assumptions the proof follows as before
and we can prove the following result where we distinguish two cases depending if $\alpha$
is small enough to allow the solution to have a proper sense. A different
formulation for the larger values of $\alpha$. We write results in the terms of $\bfF$ such that
$\bff=\Div\bfF$, but this can be translated in terms of $\bff$ only, again
using~\cite{GHH2006} and~\eqref{eq:characterization-fractional}.
%
%\marginpar{write also if you want in terms of  $\bff\in W^{s,p}$}%1+\alpha/3,3/2}(\Omega)$.}
\begin{theorem}
\label{thm:thm_existence_approximation-2}
\begin{enumerate}
\item[(a)] Let $\alpha<\frac{6}{5}$ and suppose that $\bff=\Div\bfF$ for some  $\bfF\in
  L^{3/2}(\Omega,d^{-\alpha/2})$. Then, there exists a weak solution $\bfv\in
  W^{1,3}_{0,\sigma}(\Omega,d^\alpha)$ of the problem~\eqref{eq:BL-steady} such that
  \begin{equation*}
    \int_\Omega(\nl)%(\nabla\bfv)\bfv
\cdot\bfphi+%\,\mathrm{d}\bfx+\int_\Omega
    d^\alpha\,|\bfomega|\bfomega\cdot\curl\bfphi\,\mathrm{d}\bfx=-\int_\Omega\bfF
    \cdot\nabla\bfphi\,\mathrm{d}\bfx\qquad\forall\,\bfphi\in  C^{\infty}_{0,\sigma}(\Omega), %W^{1,3}_{0,\sigma}(\Omega,d^\alpha),
  \end{equation*}
and 
\begin{equation*}
  \int_\Omega
    d^\alpha\,|\bfomega|^3\,\mathrm{d}\bfx=-\int_\Omega\bfF
    \cdot\nabla\bfv \,\mathrm{d}\bfx.
\end{equation*}
\item[(b)] Let $\frac{6}{5}\leq \alpha<\frac{3}{2}$ and suppose that  $\bff=\Div\bfF$ with $\bfF\in
L^{3/2}(\Omega,d^{-\alpha/2})$. Then, there exists a weak solution $\bfv\in
W^{1,3}_{0,\sigma}(\Omega,d^\alpha)$ of the problem~\eqref{eq:BL-steady} such that
  \begin{equation*}
    \int_\Omega(\nl)%abla\bfv)\bfv
\cdot\bfphi+%\int_\Omega
    d^\alpha\,|\bfomega|\bfomega\cdot\curl\bfphi\,\mathrm{d}\bfx=-\int_\Omega\bfF
    \cdot\nabla\bfphi\,\mathrm{d}\bfx\qquad\forall\,\bfphi\in C^{\infty}_{0,\sigma}(\Omega).%W^{1,s'}_{0,\sigma}(\Omega),
  \end{equation*}
%  where $s'>\frac{9}{6-4\alpha}$ is arbitrary.
  \end{enumerate}
\end{theorem}
\begin{proof}
  The proof follows exactly the same lines of that of Theorem~\ref{thm:existence}. We
  observe that in order to use $\bfv$ itself as test function, hence to cancel the convective
  term, we need for instance the estimate
  \begin{equation*}
   \Big|
   \int_{\Omega}(\nl)\cdot\bfv\,\mathrm{d}\bfx\Big|\leq\|\bfomega\|_{3/(1+\alpha)-\varepsilon}\|\bfv\|_{9/\alpha}^{2}
   \quad\text{for  some }\varepsilon>0,
  \end{equation*}
which holds true  if  $\frac{1+\alpha}{3}+\frac{2\alpha}{9}<1$ or, equivalently, if  $\alpha<\frac{6}{5}$.

In the other case, the convective term is still in $L^{1}(\Omega)$, but the function
$\bfv$ is not regular enough to be used globally as test function and to write the
energy-type estimate.
\end{proof}

\bigskip

We consider now even larger values of $\alpha$ and we observe that for all $0<\alpha<2$ it
holds true that,
\begin{equation*}
  \bfv\otimes\bfv \in L^{\frac{9}{2\alpha}}(\Omega)\subset
  L^{\frac{9}{4}}(\Omega)\subset L^{1}(\Omega),
\end{equation*}
hence, we can reformulate the problem with the convective term written as follows
\begin{equation*}
  (\nabla\bfv)\bfv=\Div(\bfv\otimes\bfv),
\end{equation*}
and consider the following notion of weak solution
\begin{definition}
  \label{def:weak3}
  We say that $\bfv\in W^{1,3}_{0,\sigma}(\Omega,d^\alpha)$ is a weak solution
  to~\eqref{eq:BL-steady-generalized} if
  \begin{equation*}
    -  \int_\Omega\bfv\otimes\bfv:\nabla\bfphi+%\,\mathrm{d}\bfx+\int_\Omega
    d^\alpha\,|\bfomega|\bfomega\cdot\curl\bfphi\,\mathrm{d}\bfx=
    -\int_\Omega\bfF\cdot\nabla\bfphi\,\mathrm{d}\bfx\qquad\forall\,\bfphi\in  
    C^\infty_{0,\sigma}(\Omega).%     W^{1,\frac{9\alpha}{9\alpha-2}}_{0,\sigma}(\Omega). 
  \end{equation*}
\end{definition}
A similar argument can be used also to prove the following result, changing the notion of
weak solution.
 \begin{theorem}
\label{thm:thm_existence_approximation-3}
Let $0\leq \alpha<2$ and suppose that $\bff=\Div\bfF$ with $\bfF\in
L^{3/2}(\Omega,d^{-\alpha/2})$. Then, there exists a weak solution $\bfv\in
W^{1,3}_{0,\sigma}(\Omega,d^\alpha)$ of the problem~\eqref{eq:BL-steady} in the sense of
Definition~\ref{def:weak3}.
\end{theorem}
\begin{remark}
  The same reasoning can be used to handle the problem~\eqref{eq:1.1.1} below with
  $\frac{9}{5}<p<3$ and any $\alpha<p-1$. The important observation is that we still have
  $\bfv \in W^{1,p}_{\sigma}(K)$ for all $K\Subset \Omega$ and hence $\bfv\otimes\bfv \in
  L^{p^{*}/2}_{Loc}(\Omega)$. The convergence of the nonlinear stress tensor follows in
  the same way as before as well.
\end{remark}
\subsection{Extension to  values of $p$ smaller than $\frac{9}{5}$}
We now study what happens in the case of  a model with smaller values of $p$, hence we
consider the generic system
\begin{align}
\label{eq:1.1.1}
  \left\{\begin{aligned} 
\Div(\bfv\otimes\bfv)%      (\nabla\bfv)\bfv
+\curl\big(d^\alpha|\bfomega|^{p-2}\bfomega\big)
      +\nabla \pi&=\Div\bfF& \mbox{in        $\Omega$,}
      \\
      \Div \bfv&=0\quad& \mbox{in $\Omega$,}
      \\
      \bfv&=\mathbf{0}\quad\quad& \mbox{ \,on $\partial\Omega$,}
\end{aligned}\right.
\end{align}
with $1<p<3$ and $0\leq\alpha<p-1$.
\begin{definition}
  \label{def:weak4}
We say that $\bfv\in W^{1,p}_{0,\sigma}(\Omega,d^\alpha)$ is a weak solution
to~\eqref{eq:1.1.1} if
  \begin{equation*}
  -  \int_\Omega\bfv\otimes\bfv:\nabla\bfphi+%\,\mathrm{d}\bfx+\int_\Omega
    d^\alpha\,|\bfomega|^{p-2}\bfomega\cdot\curl\bfphi\,\mathrm{d}\bfx=-\int_\Omega\bfF\cdot\nabla\bfphi\,\mathrm{d}\bfx
    \qquad\forall\,\bfphi\in 
C^\infty_{0,\sigma}(\Omega).%    W^{1,\frac{9\alpha}{9\alpha-2}}_{0,\sigma}(\Omega). 
  \end{equation*}
\end{definition}
We obtain the following result
\begin{theorem}
  Let $p>\frac{6}{5}$, $0\leq \alpha<p-1$, and suppose that $\bff=\Div\bfF$ with $\bfF\in
L^{p'}(\Omega,d^{-\alpha/(p-1)})$. Then, there exists a weak solution $\bfv\in
W^{1,p}_{0,\sigma}(\Omega,d^\alpha)$ of the problem~\eqref{eq:1.1.1} in the sense of
Definition~\ref{def:weak4}.
\end{theorem}
\begin{proof}
%\marginpar{Enlarge a little this part}  
  As before in the previous proofs we regularize~\eqref{eq:1.1.1} and consider the system
  \begin{equation}
    \left\{\begin{aligned} 
        -\epsilon\,\Div|\bfD\bfve|^{p-2}\bfD\bfve+
        (\nabla\bfve)\bfve+\hspace{1cm}&&
        \\
        \curl\big(d^\alpha|\bfoe|^{p-2}\bfoe\big) 
        +\nabla \pi&=\Div\bfF& \mbox{in        $\Omega$,}
        \\
        \Div \bfve&=0\quad& \mbox{in $\Omega$,}
        \\
        \bfve&=\mathbf{0}\quad\quad& \mbox{ \,on $\partial\Omega$,}
      \end{aligned}\right.\label{eq:approximationp}
  \end{equation}
  and we can follow the same procedure to prove existence of the approximate system, at
  least for $p>6/5$, following the approach from M\'alek and Steinhauer \textit{et
    al.}~\cite{DMS,FMS}.  
  Also, we obtain uniform estimate
    \begin{equation*}
%    \label{eq:apriori-epsilon2'}
    \epsilon\|\bfve\|_{W^{1,p}}^3+\int_\Omega d^{\alpha}\,|\bfoe|^p\,\mathrm{d}\bfx\leq
    C(\Omega,\bfF), 
\end{equation*}
   which yields
    \begin{align}
    \bfve&\rightharpoonup \bfv\qquad\text{in }W^{1,q}_{0,\sigma}(\Omega)\qquad
    \forall\,q<\frac{p}{\alpha+1}\label{eq:A}
    \\
    \bfve&\rightarrow \bfv\qquad\text{in
    }L^{r}_{\sigma}(\Omega)\qquad\forall\,r<\frac{3p}{3\alpha+3-p}\label{eq:B}
    \\
    \bfve&\rightarrow \bfv\qquad \text{a.e. in }\Omega,\label{eq:C}
    \\
    \varepsilon|\bfD\bfve|^{p-2}\bfD\bfve&\rightarrow \mathbf{0}\qquad\text{in
    }L^{p'}(\Omega).\label{eq:D}\\
\label{eq:E}
       (\nabla \bfve)_{|K}&\rightharpoonup \nabla\bfv_{|K}\qquad\text{in }L^{p}(K)\qquad
      \forall\,K\Subset\Omega,
      \\
      (\bfve)_{|K}&\rightarrow \bfv_{|K}\qquad\text{in }L^{r}(K)\qquad\forall\,r<\frac{3p}{3-p}.\label{eq:F}
  \end{align}
%
%\subsection{New from Dominic}
%%
%We consider the approximate Galerkin system 
%\begin{equation}
%  \left\{\begin{aligned} 
%      -\frac{1}{m}\Div|\bfD\bfv^{m}|^{p-2}\bfD\bfv^{m}+
%      (\nabla\bfv^{m})\bfv^{m}+\curl\big(d^\alpha|\bfomega^{m}|^{p-2}\bfomega^{m}\big) 
%      +\nabla \pi&=\bff& \mbox{in        $\Omega$,}
%      \\
%      \Div \bfv^{m}&=0\quad& \mbox{in $\Omega$,}
%      \\
%      \bfv^{m}&=\marthbf{0}\quad\quad& \mbox{ \,on $\partial\Omega$.}
%    \end{aligned}\right.\label{eq:approximationpm}
%\end{equation}
%
%while passing to the limit in the nonlinear term requires a different approach. 
  Based on the previous observations we obtain the limit system
  \begin{align*}
%\label{eq:limit} 
    \left\{
      \begin{array}{rc} \Div(\bfv\otimes\bfv)+\curl \overline\bfS+\nabla \pi=\Div\bfF&
        \mbox{in $\Omega$,}
        \\
        \Div \bfv=0\qquad& \mbox{in $\Omega$,}
        \\
        \bfv=0\qquad\quad& \mbox{ \,on $\partial\Omega$,}
      \end{array}
    \right.
  \end{align*}
  where the first equation has to be understood in the sense of distributions.
  Here the limit is taken along some sequence $\epsilon_{m}\to0$ and for simplicity we set
  \begin{equation*}
    \bfv^{m}:=\bfv_{\epsilon_{m}}\qquad \text{and}\qquad \bfon:=\bfomega_{\epsilon_{m}}.
  \end{equation*}
  Here $\overline\bfS$ denotes the weak limit of $d^\alpha|\bfon|^{p-2}\bfon$ which exists
  in $L^{p'}_{loc}(\Omega)$. The remaining effort is to show that
  $\overline\bfS=d^\alpha|\bfomega|\bfomega$, i.e.
  \begin{align}
    \label{eq:n3.18}
    \big\langle d^{\alpha}|\bfon|^{p-2}\bfon, \ep(\bfphi)\big\rangle & \to
    \big\langle d^{\alpha}|\bfom|^{p-2}\bfom, \ep(\bfphi)\big\rangle \quad\forall\,\bfphi
    \in C_{0,\sigma}^\infty(\Omega).
  \end{align}
  It suffices to prove that $\bfon \to \bfom$ almost everywhere.  This follows from the strict
  monotonicity of the operator $\bfxi\mapsto|\bfxi|^{p-2}\bfxi$ provided that for a
  certain $\theta\in (0, 1]$ and every ball $B \subset \Omega$ with $4B \subset \Omega$
  \begin{align}
    \label{eq:n3.20}
    \limsup_{m\to\infty}\int_B \Big( |\bfon|^{p-2}\bfon - |\bfom|^{p-2}\bfom) \cdot( \bfon
    - \bfom ) \Big)^{\theta} \, d\bfx = 0 \,.
  \end{align}
  To verify equation~\eqref{eq:n3.20}, let $\eta \in C^\infty_0(2B)$ be as
  in~\eqref{eq:bump}, with $B$ now such that $4B\Subset\Omega$. Define 
  \begin{align*}
    \bfw^{m} := \eta\, (\bfv^{m} - \bfv) - \Bog_{2B}( \nabla \eta \cdot (\bfv^{m} - \bfv)),
  \end{align*}
  where $\Bog_{2B}$ is the \Bogovskii{} operator on $2B$ from $L^p_0(2B)$ to
  $W^{1,p}_0(2B)$. Since $\nabla \eta \cdot (\bfv^{m} - \bfv)$ is bounded in $L^p_0(2B)$
  by~\eqref{eq:F}, we have that $\bfw^{m}$ is bounded in
  $W^{1,p}_{0,\sigma}(2B)$. Moreover, $\bfv^{m} \to \bfv$ in $L^2(2B)$ and the continuity
  of $\Bog_{2B}$ implies $\bfw^{m} \to \mathbf{0}$ at least in $L^1(2B)$. In particular,
  we can apply the solenoidal Lipschitz truncation of Theorem~\ref{thm:remlip} to
  construct a suitable double sequence $\bfw^{m,j} \in W^{1,\infty}_{0,\sigma}(4B)$.

  We use now $\bfw^{m,j}$ as a test function in~\eqref{eq:approximationp} and obtain
  \begin{align*}
    \begin{aligned}
      \langle d^{\alpha} |\bfon|^{p-2}\bfon &-d^{\alpha} |\bfom|^{p-2}\bfom, \curl(\bfw^{m,j})\rangle =
      - \langle d^\alpha|\bfom|^{p-2}\bfom, \curl(\bfw^{m,j})\rangle
      \\
      &\hphantom{=} - \epsilon_{m} \langle|\bfD\bfv^m|^{p-2}\bfD\bfv^{m}, \bfD\bfw^{m,j}) \rangle + \langle
      \bfF,\nabla\bfw ^{m,j} \rangle
      \\
      &\hphantom{=} + \langle\bfv^m \otimes \bfv^m, \nabla\bfw^{m,j}\rangle.
    \end{aligned}
  \end{align*}
  It follows from the properties of $\bfw^{m,j}$ and $\bfv^m$ that the right-hand side
  converges for fixed $j$ to zero as $m\rightarrow\infty$. So we get
  \begin{align*}
%    \label{eq:44}
    \lim_{m \to \infty} \langle d^\alpha |\bfon|^{p-2}\bfon - d^\alpha|\bfom|^{p-2}\bfom,
    \curl(\bfw^{m,j})\rangle =0. 
  \end{align*}

  We decompose the set $4B$ into $\set{\bfw^{m}\not= \bfw^{m,j}}$ and $4B \cap
  \set{\bfw^{m}= \bfw^{m,j}}$ to get
  \begin{align*}
    (I) &:= \biggabs{\int\limits_{4B \cap \set{\bfw^n = \bfw^{m,j}}}
      \mspace{-40mu} \eta\, d^{\alpha}\,\Big( |\bfon|^{p-2}\bfon - |\bfom|^{p-2}\bfom\Big) \cdot (\bfon
      - \bfom)\,\mathrm{d}\bfx}
    \\
    &= \biggabs{\int\limits_{\set{\bfw^n \not= \bfw^{m,j}}}
      \mspace{-25mu} d^{\alpha}\Big( |\bfon|^{p-2}\bfon - |\bfom|^{p-2}\bfom\Big) \cdot
      \curl(\bfw^{m,j})\,\mathrm{d}\bfx}
    \\
    &+  \biggabs{\int\limits_{4B \cap \set{\bfw^n =
          \bfw^{m,j}}} \mspace{-30mu}d^{\alpha}\Big( |\bfon|^{p-2}\bfon - |\bfom|^{p-2}\bfom\Big)
      \cdot \big( \nabla \eta \times (\bfv^{m} - \bfv) \big)\,\mathrm{d}\bfx}
    \\
    &+ \biggabs{\int\limits_{4B \cap \set{\bfw^n =
          \bfw^{m,j}}} \mspace{-30mu} d^{\alpha}\Big( |\bfon|^{p-2}\bfon - |\bfom|^{p-2}\bfom\Big)
      \cdot \curl\big( \Bog_{2B}( \nabla \eta \cdot (\bfv^{m} - \bfv)) \big)\,\mathrm{d}\bfx}
    \\
    &=: (II) + (III) + (IV).
  \end{align*}
  Since $\nabla \eta \otimes (\bfv^{m} - \bfv) \stackrel{m}{\to} 0$ in $L^p(2B)$, we have
  $(III) + (IV) \stackrel{m} \to 0$, recall~\eqref{eq:E} and~\eqref{eq:F}. Note that we also used the
  continuity of $\Bog_{2B}$ from $L^p_0(2B)$ to $W^{1,p}_0(2B)$.

  By H\"older's inequality,~\eqref{eq:E} and Theorem~\ref{thm:remlip}-\ref{itm:remlip5}
  \begin{align*}
    (II) &\leq \limsup_{m\to+\infty} \big(\norm{\bfon}_{p'} + \norm{\bfom}_{p'}\big)\,
    \norm{\chi_{\set{\bfw^n\not= \bfw^{m,j}}}\nabla\bfw^{m,j}}_p
    \\
    &\leq \,c 2^{-j/p}\|\nabla\bfw^{m}\|_p\leq \,c 2^{-j/p}.
  \end{align*}
  Overall we get
  \begin{align*}
    \limsup_{m\to+\infty} \biggabs{\int\limits_{4B \cap \set{\bfw^{m} = \bfw^{m,j}}}
      \mspace{-40mu} \eta\, d^\alpha\,\big( |\bfon|^{p-2}\bfon - |\bfom|^{p-2}\bfom\big) \cdot(\bfon
      - \bfom)\,\mathrm{d}\bfx} \leq c\, 2^{-j/p}.
  \end{align*}
  This implies
  \begin{align*}
    \label{eq:4}
    \limsup_{m\to+\infty} \int\limits_{4B} \Big( \eta\, d^\alpha\,\big( |\bfon|^{p-2}\bfon -
    |\bfom|^{p-2}\bfom\big) \cdot(\bfon - \bfom) \Big)^\theta \,\mathrm{d}\bfx = 0
  \end{align*}
  for any $\theta\in(0,1)$ as a consequence of~\eqref{eq:E} and
  Theorem~\ref{thm:remlip}-\ref{itm:remlip5}.  Now,~\eqref{eq:n3.20} follows form $\eta
  \geq\chi_B$ and $d\geq C_B>0$ in $B$. So we obtain~\eqref{eq:n3.18} as desired, which
  finishes the proof. 
\end{proof}
\begin{remark}
  We are not considering here problems of regularity of the weak solutions and also of
  less regular weight functions as in the recent studies by Cirmi, D'Asero, and
  Leonardi~\cite{CDL2020}. Moreover, as it is the case for similar problems, uniqueness
  for the system~\eqref{eq:BL-steady} is not known, even for small enough
  solutions. Uniqueness of small solutions to~\eqref{eq:BLM} follows directly by the same
  results for the Navier-Stokes equations, as explained in Galdi~\cite{Gal2011}.  On the
  other hand uniqueness of small solutions --even for the regularized
  system~\eqref{eq:approximationp}-- is not known for $p>2$ or for $p<\frac{9}{5}$, see
  Blavier and Mikeli\'{c}~\cite{BM1995} and the review in Galdi~\cite{Gal2008}.
\end{remark}

\def\cprime{$'$} \def\ocirc#1{\ifmmode\setbox0=\hbox{$#1$}\dimen0=\ht0
  \advance\dimen0 by1pt\rlap{\hbox to\wd0{\hss\raise\dimen0
  \hbox{\hskip.2em$\scriptscriptstyle\circ$}\hss}}#1\else {\accent"17 #1}\fi}
  \def\polhk#1{\setbox0=\hbox{#1}{\ooalign{\hidewidth
  \lower1.5ex\hbox{`}\hidewidth\crcr\unhbox0}}}

\end{document}